\documentclass[11pt,english]{article}

\usepackage{geometry}
\usepackage[T1]{fontenc}
\usepackage[utf8]{inputenc}
\usepackage[english]{babel}
\usepackage{latexsym}
\usepackage{mathrsfs}
\usepackage{enumerate}
\usepackage{amsmath}
\usepackage{amssymb}
\usepackage{mathtools}
\usepackage{braket}
\usepackage{amsthm}
\usepackage{graphicx} %per le immagini
\usepackage[dvips]{psfrag}
\usepackage{tikz}
\usepackage{tikz-cd}
\usepackage{frontespizio}
\usepackage{shellesc}
\usepackage{quoting}
\usepackage{url}
\usepackage{hyperref}
\usepackage{tabularx}
\hypersetup{colorlinks=true,linkcolor=black}
\usepackage{booktabs}
\usepackage{spectralsequences}

\usepackage[autostyle,italian=guillemets]{csquotes}
\usepackage[backend=biber, style=numeric, bibstyle=numeric,doi=false,isbn=false,url=false,eprint=false]{biblatex}

\theoremstyle{plain} 
\newtheorem{thm}{Theorem}[section] 
\newtheorem{corollario}[thm]{Corollary} 
\newtheorem{lem}[thm]{Lemma} 
\newtheorem{prop}[thm]{Proposition} 

\theoremstyle{definition}
\newtheorem{defn}[thm]{Definition} 
\newtheorem{es}{Example}[thm]

\newtheorem{oss}[thm]{Remark} 

\theoremstyle{remark}

\newcommand{\M}{\mathcal{M}}
\newcommand{\numberset}{\mathbb}
\newcommand{\N}{\numberset{N}}
\newcommand{\F}{\numberset{F}}
\newcommand{\R}{\numberset{R}}
\newcommand{\C}{\numberset{C}}
\newcommand{\Z}{\numberset{Z}}
\newcommand{\Q}{\numberset{Q}}
\DeclarePairedDelimiter{\abs}{\lvert}{\rvert}

\newsavebox{\pullback}
\sbox\pullback{%
	\begin{tikzpicture}%
		\draw (0,0) -- (1ex,0ex);%
		\draw (1ex,0ex) -- (1ex,1ex);%
\end{tikzpicture}}

\newsavebox{\pushout}
\sbox\pushout{%
	\begin{tikzpicture}%
		\draw (0,0) -- (0ex,1ex);%
		\draw (0ex,1ex) -- (1ex,1ex);%
\end{tikzpicture}}

\newsavebox{\hpushout}
\sbox\hpushout{%
	\begin{tikzpicture}%
		\draw (0ex,2ex) -- (0ex,3ex);%
		\draw (0ex,3ex) -- (1ex,3ex);%
		\draw (-2ex,5ex) node [anchor=north west][inner sep=0.75pt]  {$h$};
		
\end{tikzpicture}}

\title{Homology operations for gravity algebras}
\date {}
\author{Tommaso Rossi\footnote{ tommaso.rossi118@gmail.com. This work was funded by the PhD program of the University of Roma Tor Vergata and by the MIUR Excellence Department Project awarded to the Department of Mathematics,
		University of Roma Tor Vergata, CUP E83C18000100006.} \\
		\footnotesize Dipartimento di Matematica, Università di Roma Tor Vergata}
\begin{document}
	
	\maketitle
	
	\begin{abstract}
		Let $\M_{0,n+1}$ be the moduli space of genus zero Riemann surfaces with $n+1$ marked points. In this paper we compute $H_*^{\Sigma_n}(\M_{0,n+1};\F_p)$ and $H_*^{\Sigma_n}(\M_{0,n+1};\F_p(\pm 1))$ for any $n\in\N$ and any prime $p$, where $\F_p(\pm 1)$ denotes the sign representation of the symmetric group $\Sigma_n$. The interest in these homology groups is twofold: on the one hand classes in these equivariant homology groups parametrize homology operations for gravity algebras. On the other hand the homotopy quotient $(\M_{0,n+1})_{\Sigma_n}$ is a model for the classifying space for $B_n/Z(B_n)$, the quotient of the braid group $B_n$ by its center.
	\end{abstract}

	\tableofcontents
	
	\section{Introduction}
	
	The Gravity operad $Grav$ was introduced by Getzler in \cite{Getzler94} and \cite{Getzler} as a sub-operad of $H_*(\mathcal{D}_2)$, where $\mathcal{D}_2$ is the little two disk operad. The space of arity $n$ operations is given by $sH_*(\M_{0,n+1};\Z)$, the (shifted) homology of the moduli space of genus zero marked curves. There are also chain model versions of the Gravity operad, described in the paper \cite{Westerland} by Westerland and in \cite{Getzler-Kapranov} by Getzler-Kapranov. A comparison between these definitions has been written by Dupont and Horel in \cite{Dupont-Horel}. Remarkable examples of algebras over this operad are:
	\begin{itemize}
		\item $H_*^{S^1}(X)$, where $X$ is an algebra over the framed little two disk operad $f\mathcal{D}_2$ (see \cite{Westerland}).
		\item If $M$ is a closed oriented manifold, the string topology operations on $s^{1-d}H_*^{S^1}(LM)$ assemble to a gravity algebra structure (see \cite{Chas-Sullivan} and \cite{Westerland}).
		\item If $A$ is a Frobenius algebra, then $HC^*(A)$ is a gravity algebra. Further examples along these lines can be found in \cite{Ward}.
	\end{itemize}
	%In this section we recall the main facts about this operad and the algebras over it. There are also chain model versions of the Gravity operad, described By Westerland \cite{Westerland} and Getzler-Kapranov \cite{Getzler-Kapranov}. A comparison between these definitions has been written by Dupont and Horel in \cite{Dupont-Horel}. 
	In the theory of $H_*(\mathcal{D}_2)$-algebras a key role is played by the Dyer-Lashof operations, which correspond to classes in  $H_*^{\Sigma_n}(\mathcal{D}_2(n);\mathbb{F}_p)$ and $H_*^{\Sigma_n}(\mathcal{D}_2(n);\mathbb{F}_p(\pm 1))$ (here $\mathbb{F}_p(\pm 1)$ denotes the sign representation of the symmetric group $\Sigma_n$). The knowledge of these operations is crucial to describe, for example, the homology of the braid groups $B_n$ and of  $\Omega^2\Sigma^2X$, as explained in the remarkable work of F. Cohen \cite{Cohen}. Similarly, classes in $H_*^{\Sigma_n}(\mathcal{M}_{0,n+1};\mathbb{F}_p)$ and $H_*^{\Sigma_n}(\mathcal{M}_{0,n+1};\mathbb{F}_p(\pm 1))$ give rise to homology operations for Gravity algebras. More precisely: given a gravity algebra $(A,d_A)\in Ch(\F_p)$ and a class $Q\in H_*^{\Sigma_n}(\mathcal{M}_{0,n+1};\F_p)$ (resp. $Q\in H_*^{\Sigma_n}(\mathcal{M}_{0,n+1};\F_p(\pm 1)$) we can construct an operation
	\[
	Q:H_*(A)\to H_{n*+\abs{Q}+1}(A)
	\]
	which acts on even (resp. odd) degree classes. Another motivation one might have to study these equivariant homology groups came from group homology: indeed the homotopy quotient $(\M_{0,n+1})_{\Sigma_n}$ is a model for the classifying space of $B_n/Z(B_n)$, the quotient of the braid group $B_n$ by its center. Therefore we have
	\begin{align*}
	&H_*^{\Sigma_n}(\mathcal{M}_{0,n+1};\mathbb{F}_p)=H_*(B_n/Z(B_n);\F_p)\\  &H_*^{\Sigma_n}(\mathcal{M}_{0,n+1};\mathbb{F}_p(\pm 1))=H_*(B_n/Z(B_n);\F_p(\pm 1))
	\end{align*}
	In this paper we compute these homology groups for any $n\in\mathbb{N}$ and any prime number $p$. In the case of the trivial representation the main result is the following (for a more precise statement see Theorem \ref{thm: calcolo dell'omologia equivariante degli spazi di configurazione quando p divide n} and Theorem \ref{thm:calcolo omologia equivariante quando p non divide n o n-1}):
	\begin{thm}\label{thm: teorema 1 intro}
	Let $n\in \N$ and $p$ a prime number. Then:
	\begin{itemize}
		\item If $n=0,1$ mod $p$ we have $H_*(B_n/Z(B_n);\F_p)\cong H_*(B_n;\F_p)\otimes \F_p[c]$, where $c$ is a variable of degree two.
		\item Otherwise $H_*(B_n/Z(B_n);\F_p)$ is described explicitly as a quotient of $H_*(B_n;\F_p)$.
	\end{itemize}
\end{thm}
To state the result in the case of the sign representation we need some preliminary notation. 	Let $p$ be an odd prime. $H_*(B_n/Z(B_n);\F_p(\pm 1))$ can be identified (up to a shift) with a subspace of 
\[
\Lambda(x,Q^1(x),Q^2(x),\dots)\otimes \F_p[\beta Q^1(x),\beta Q^2(x),\dots]\otimes \F_p[c]
\]
where $x$ is a variable of degree one, $Q^i(x)$ has degree $2p^i-1$, $\beta Q^i(x)$ has degree $2p^i-2$ and $c$ has degree two. In order to describe an additive basis of  $H_*(B_n/Z(B_n);\F_p(\pm 1))$ we need a weight grading $w$, which is generated by the following rules:
\begin{itemize}
	\item $w(1)=0$
	\item $w(x)=1$
	\item $w(ab)=w(a)+w(b)$
	\item $w(Q^i(x))=w(\beta Q^i(x))=p^i$
\end{itemize} 
 Now we are ready to describe a basis for $H_*(B_n/Z(B_n);\F_p(\pm 1))$ (Theorem \ref{thm: identificazione esplicita dell'omologia twistata}):
\begin{thm}\label{thm: teorema 2 intro}
	Let $p$ be an odd prime, $n\geq 2$. We will denote by $s$ the degree shift by $1$. Then $s^{n}H_*(B_n/Z(B_n);\F_p(\pm 1))$ is generated by elements of the form
	\[
	a\otimes c^i \qquad i\geq 0
	\] 
	where $a\in \Lambda(x,Q^1(x),Q^2(x),\dots)\otimes \F_p[\beta Q^1(x),\beta Q^2(x),\dots]$ is a monomial of weight $n$. 
\end{thm}
We end this introduction by giving an overview of the techniques we used to prove these results.
	\paragraph{Strategy of the proof (Theorem \ref{thm: teorema 1 intro}):}
	The key observation to prove Theroem \ref{thm: teorema 1 intro} is the following:  the homotopy quotients  $(\mathcal{M}_{0,n+1})_{\Sigma_n}$ and $C_n(\C)_{S^1}$ (where $C_n(\C)$ is the unordered configuration space) are both models for the classifying space of $B_n/Z(B_n)$. This allows us to do the computation of $H_*^{\Sigma_n}(\M_{0,n+1};\F_p)$ by looking at the Serre spectral sequence associated to the fibration
	\begin{equation}\label{fibration 1}
		\begin{tikzcd}
			& C_n(\C)\arrow[r,hook] & C_n(\C)_{S^1} \arrow[r] & BS^1   
		\end{tikzcd}
	\end{equation}
	which is much simpler than
	\[
	\begin{tikzcd}
		& \M_{0,n+1} \arrow[r,hook] & (\M_{0,n+1})_{\Sigma_n} \arrow[r] & B\Sigma_n   
	\end{tikzcd}
	\]
	because in the first case there is not any monodromy. Moreover, in the (homological) Serre spectral sequence associated to (\ref{fibration 1}) the homology of the fiber is well known, and the differential of the second page is given by the BV-operator $\Delta$ (Proposition \ref{prop: differenziale dato da delta}). After a careful analysis of this spectral sequence we get the result (Section \ref{sec: operations for even classes}).

\paragraph{Strategy of the proof (Theorem \ref{thm: teorema 2 intro}):}
	The computation of $H_*^{\Sigma_n}(B_n/Z(B_n);\mathbb{F}_p(\pm 1))$ involves different methods, based on the work of Cohen, B\"odigheimer and Peim \cite{Cohen-Bodigheimer}. To explain the strategy we need some notation: let $\lambda:E\to B$ be a fiber bundle with fiber $F$ and consider the (ordered) fiberwise configuration space
	\[
	E(\lambda,n)\coloneqq\{(e_1,\dots,e_n)\in E^n\mid e_i\neq e_j \text{ and } \lambda(e_i)=\lambda(e_j) \text{ if } i\neq j\}
	\]
	Now let $X$ be a connected CW-complex with basepoint $\ast$. The space of \emph{fiberwise configurations with label in} $X$ is defined as
	\begin{equation*}
		E(\lambda;X)\coloneqq \bigsqcup_{n=0}^{\infty}E(\lambda,n)\times_{\Sigma_n}X^n/\sim 
	\end{equation*}
	where $\sim$ is the equivalence relation determined by 
	\[
	(e_1,\dots,e_n)\times(x_1,\dots,x_n)\sim (e_1,\dots,\hat{e}_i,\dots,e_n)\times(x_1,\dots,\hat{x}_i,\dots,x_n)
	\]
	when $x_i=\ast$. Now the idea is the following: if $\C\hookrightarrow E\xrightarrow{\lambda} \C P^{\infty}$ is the tautological line bundle, then $E(\lambda;n)/\Sigma_n$ is a model for the classifying space of $B_n/Z(B_n)$. In this situation one can prove that $H_*(B_n/Z(B_n);\F_p(\pm 1))$ can be described as a certain subspace of $H_*(E(\lambda;S^{2q+1});\F_p)$ (Proposition \ref{prop: interpretazione come fiberwise configurations}). Therefore it suffices to compute  $H_*(E(\lambda;S^{2q+1});\F_p)$, and this is done by looking at the fibration
	\[
	C(\C;S^{2q+1})\hookrightarrow E(\lambda;S^{2q+1})\to \C P^{\infty}
	\]
	where $C(\C;S^{2q+1})$ is the configuration space of points in the plane with labels in $S^{2q+1}$. The result is the following (Theorem \ref{thm:collasso sequenza spettrale per i labelled fiberwise config. spaces}):
	\begin{thm}
		For any $q\in \N$ and $p$ a prime number 
		\[
		H_*(E(\lambda;S^{2q+1});\F_p)\cong H_*(C(\C;S^{2q+1});\F_p)\otimes H_*(\C P^{\infty};\F_p)
		\]
	\end{thm}
	In Paragraph \ref{sec: examples} we finally describe an additive basis of $H_*^{\Sigma_n}(B_n/Z(B_n);\mathbb{F}_p(\pm 1))$ by means of basis of $ H_*(C(\C;S^{2q+1});\F_p)$ and of $ H_*(\C P^{\infty};\F_p)$.
	\paragraph{Plan for the paper:} Section \ref{sec: gravity operad} is a recollection of well known facts about the gravity operad. In Section \ref{sec: homology operations for gravity algebras} we explain the connection between classes in $H_*^{\Sigma_n}(\M_{0,n+1};\F_p)$ (or in $H_*^{\Sigma_n}(\M_{0,n+1};\F_p(\pm 1))$) and homology operations for gravity algebras. Section \ref{sec: operations for even classes} and Section \ref{sec: operations for odd degree classes} contain the computation of these homology groups.
	\paragraph{Future directions:} the Dyer-Lashof operations satisfy a bunch of relations (e.g. Cartan relations, Adem relations). It would be interesting to understand what kind of relations we get in the context of gravity algebras. This would imply a complete knowledge of the homology of the free gravity algebras. Moreover it would be interesting to see if the computation of $H_*^{\Sigma_n}(\M_{0,n+1};\F_p)$ can be used to get some information about $H_*^{\Sigma_n}(\overline{\M}_{0,n+1};\F_p)$ by mean of the morphism $H_*^{\Sigma_n}(\M_{0,n+1};\F_p)\to H_*^{\Sigma_n}(\overline{\M}_{0,n+1};\F_p)$ induced by the inclusion.
	\paragraph{Acknowledgments:} This project is part of the author's PhD thesis, funded by the PhD program of the University of Roma Tor Vergata and by the MIUR Excellence Project MatMod@TOV awarded to the
	Department of Mathematics, University of Rome Tor Vergata,  CUP E83C18000100006. I can not omit the invaluable help that my advisor Paolo Salvatore gave to me in these years of PhD. I am grateful to him for both proposing me this topic and for sharing with me many good ideas about it. Without his guidance and support this project would have not be possible. I also thank Lorenzo Guerra for many enlightening discussions about the cohomology of groups and the anonymous referee for pushing me to make the computation of $H_*^{\Sigma_n}(B_n/Z(B_n);\F_p(\pm 1))$ more explicit. 
	
	\section{The Gravity operad}\label{sec: gravity operad}
	Let $X$ be a $S^1$-space. The action $\theta:S^1\times X\to X$ induces an operator $\Delta:H_*(X;\Z)\to H_{*+1}(X;\Z)$ by the composition
	\[
	\begin{tikzcd}
		& H_*(X)\arrow[r] &H_*(S^1)\otimes H_*(X)\arrow[r,"\times"]&H_*(S^1\times X)\arrow[r,"\theta_*"]& H_*(X)
	\end{tikzcd}
	\]
	where the first map take a class $x\in H_*(X)$ and send it to $[S^1]\otimes x$. We will call $\Delta$ the \textbf{BV-operator} (see \cite{Getzler93}). In what follows all the homology groups are taken with integer coefficients, unless otherwise stated. To ease the notation we sometimes write $H_*(X)$ instead of $H_*(X;\Z)$.
	\begin{defn}[Getzler, \cite{Getzler93}]
		Let $\mathcal{D}_2$ be the little two disk operad. $S^1$ acts on $\mathcal{D}_2(n)$ by rotations, so we get a BV-operator $\Delta:H_*(\mathcal{D}_2(n))\to H_{*+1}(\mathcal{D}_2(n))$. This map is compatible with the operadic structure and induces a morphism of operads $\Delta:H_*(\mathcal{D}_2)\to H_{*+1}(\mathcal{D}_2)$. The kernel of this map is a sub-operad of $H_*(\mathcal{D}_2)$, called the \textbf{Gravity operad}. We will denote it by $Grav$.
	\end{defn}
	\begin{oss}
		$S^1$ acts freely on $\mathcal{D}_2(n)$ so we can identify the kernel of $\Delta:H_*(\mathcal{D}_2(n))\to H_{*+1}(\mathcal{D}_2(n))$ with $sH_*^{S^1}(\mathcal{D}_2(n))\cong s H_*(\M_{0,n+1})$, where this last isomorphism holds because $\M_{0,n+1}$ and $\mathcal{D}_2(n)/S^1$ are homotopy equivalent. To sum up we have the following identification:
		\[
		Grav(n)= sH_*(\M_{0,n+1})
		\]
	\end{oss}
	\begin{oss}
		The action of $\Sigma_{n+1}$ on $\M_{0,n+1}$ by relabelling the points induces an action in homology, making $Grav$ a cyclic operad.
	\end{oss}
	Unlike many familiar operads, the Gravity operad is not generated by a finite number of operations. However, it has a nice presentation with infinitely many generators:
	\begin{thm}[Getzler \cite{Getzler94}]
		As an operad $Grav$ is generated by (graded) symmetric operations of degree one
		\[
		\{a_1,\dots,a_n\}\in Grav(n) \quad \text{ for } n\geq 2
		\]
		Geometrically, $\{a_1,\dots,a_n\}$ corresponds to the generator of $H_0(\mathcal{M}_{0,n+1},\Z)$. These operations (called brackets) satisfy the so called generalized Jacobi relations: for any $k\geq 2$ and $l\in \N$
		\begin{equation}\label{eq: generalized jacobi relations}
			\sum_{1\leq i<j\leq k}(-1)^{\epsilon(i,j)} \{\{a_i,a_j\},a_1,\dots,\hat{a}_i,\dots,\hat{a}_j,\dots,a_k,b_1,\dots,b_l\}=
			\{\{a_1,\dots,a_k\},b_1,\dots,b_l\} 
		\end{equation}
		where the right hand term is interpreted as zero if $l=0$ and $\epsilon(i,j)=(\abs{a_1}+\dots+ \abs{a_{i-1}})\abs{a_i}+(\abs{a_1}+\dots+\abs{a_{j-1}})\abs{a_j}+ \abs{a_i}\abs{a_j}$.
	\end{thm}
	\begin{defn}
		A \textbf{Gravity algebra} (in the category of chain complexes) is an algebra over the Gravity operad. To be explicit, it is a chain complex $(A,d_A)$ together with graded symmetric chain maps $\{-,\dots,-\}:A^{\otimes k}\to A$ of degree one such that for $k\geq 3$, $l\geq 0$ and $a_1,\dots,a_k,b_1,\dots,b_l\in A$, Equation \ref{eq: generalized jacobi relations} is satisfied.
	\end{defn}
	Gravity algebras and $BV$-algebras are closely related. The idea is that every time we have a $BV$-algebra structure on the (co)homology of a space/d.g. algebra, then we get a Gravity algebra structure on the $S^1$-equivariant version of our (co)homology theory. The following list of examples should clarify this last sentence:
	\begin{enumerate}
		\item Let $X$ be a $f\mathcal{D}_2$-algebra. Then $H_*(X)$ is a $BV$-algebra and $sH_*^{S^1}(X)$ is a Gravity algebra (Westerland, \cite[Corollary 2.5]{Westerland}). More precisely, in the category of spectra the Gravity operad is defined as the homotopy fixed points $(\mathcal{D}_2)^{hS^1}$ of the little two disk operad $\mathcal{D}_2$. Westerland proved that the norm map $n:(\mathcal{D}_{2})_{bS^1}\to (\mathcal{D}_2)^{hS^1}$ is a weak equivalence of operads, where $(\mathcal{D}_{2})_{bS^1}$ is the so called \emph{transfer operad}. Corollary $2.5$ of \cite{Westerland} tells us that $(\mathcal{D}_{2})_{bS^1}$ acts on $sH_*^{S^1}(X)$ whenever $X$ is $f\mathcal{D}_2$-algebra. Combining this with the weak equivalence $(\mathcal{D}_{2})_{bS^1}\simeq (\mathcal{D}_2)^{hS^1}$ we get that $sH_*^{S^1}(X)$ is an algebra over $H_*((\mathcal{D}_{2})^{hS^1})$, i.e. a Gravity algebra.
			\item Let $M$ be an oriented $d$-dimensional manifold. The homology $H_*(LM)$ of the free loop space on $M$ carries a rich algebraic structure: the loop product of Chas-Sullivan \cite{Chas-Sullivan} endow $s^{-d} H_*(LM)$ with a commutative algebra structure. Moreover the $S^1$ action on $LM$ is compatible with this product, so $s^{-d}H_*(LM)$ is an $BV$-algebra (Cohen-Jones \cite{Cohen-Jones}). As before, if we switch to $S^1$-equivariant homology (and shift the degree appropriately) we get a Gravity algebra structure on $s^{1-d}H^{S^1}_*(LM)$ (Westerland, \cite{Westerland}).
		\item If $A$ is a Frobenius algebra, then the Hochschild cohomology $HH^*(A)$ is a $BV$-algebra. If we switch to the $S^1$-equivariant version of Hochschild cohomology (i.e. the cyclic cohomology) we get a Gravity algebra structure on $HC^*(A)$. See the paper by Ward \cite{Ward} for further details and examples.
		
	\end{enumerate}
	The following table summarizes what we said in this section:
	\[
	\begin{tabularx}{\textwidth}{lXXX}
		\toprule
		& & $BV$-algebra & $Grav$-algebra\\
		\toprule& $X$ $f\mathcal{D}_2$-algebra & $H_*(X)$  & $H_*^{S^1}(X)$\\
		\midrule
		& $M$ closed oriented & $s^{-d}H_*(LM) $ & $s^{1-d}H_*^{S^1}(LM)$\\ 
		\midrule
		& $A$ Frobenius algebra & $HH^*(A)$ & $HC^*(A)$\\
		\bottomrule
	\end{tabularx}
	\]
	
	\section{Homology operations for gravity algebras}\label{sec: homology operations for gravity algebras}
	In this section we show that classes in $H_*^{\Sigma_n}(\mathcal{M}_{0,n+1};\mathbb{F}_p)$ and $H_*^{\Sigma_n}(\mathcal{M}_{0,n+1};\mathbb{F}_p(\pm 1))$ give rise to homology operations for Gravity algebras. 
	\subsection{Equivariant operations}\label{sec: equiavariant operations}
	Let $p$ be a prime and use $\F_p$-coefficients for (co)homology from now on. Fix $grav$ any chain model for the Gravity operad. 
	\begin{oss}
		We can suppose that $\Sigma_n$ acts freely on $grav(n)$ for any $n\in\N$. Indeed it suffices to replace our chain model for the gravity operad with a cofibrant replacement of it. From now on we assume that we are in this situation.
	\end{oss}
	Now let $(A,d_A)$ be a Gravity algebra in the category $Ch(\F_p)$. The structure maps $grav(n)\otimes A^{\otimes n}\to A$ are $\Sigma_n$-equivariant, so they factor through the coinvariants:
	\[
	\begin{tikzcd}
		&grav(n)\otimes A^{\otimes n}\arrow[r] \arrow[d]& A\\
		&grav(n)\otimes_{\Sigma_n} A^{\otimes n}\arrow[ur,"\gamma"]   
	\end{tikzcd}
	\]
	Passing to homology we get
	\[
	\gamma_*:H_*(grav(n)\otimes_{\Sigma_n} A^{\otimes n})\to H_*(A)
	\]
	\begin{oss}\label{oss: identificazione omologia di A e A nella def delle gravity operations}
		Since we are working with coefficients in a field, we can define a quasi-isomorphism $H_*(A)\to A$ by choosing a basis for $H_*(A)$ and assigning to each element of it  a representative cycle in $Z_*(A)\subseteq A$. Therefore we get an isomorphism between $H_*(grav(n)\otimes_{\Sigma_n}H_*(A)^{\otimes n})$ and $H_*(grav(n)\otimes_{\Sigma_n}A^{\otimes n})$. 
	\end{oss}
	\begin{oss}\label{oss: coinvarianti calcolano omologia equivariante}
		Since $\Sigma_n$ acts freely on $grav(n)$, the coinvariants $grav(n)_{\Sigma_n}$ compute (up to a degree shift) the $\Sigma_n$-equivariant homology of $\M_{0,n+1}$. More explicitly,
		\[
		H_*(grav(n)_{\Sigma_n})\cong sH_*^{\Sigma_n}(\M_{0,n+1};\F_p)
		\]
		%It is not hard to see that the homotopy quotient $(\M_{0,n+1})_{\Sigma_{n}}$ is the classifying space of $B_n/Z(B_n)$ (see Lemma \ref{lem:equivalenza omotopica tra quozienti omotopici}) and Proposition \ref{prop: quizienti omotopici come spazi classificanti}), the quotient of the braid group on $n$-strands by its center. So $H_*^{\Sigma_n}(\M_{0,n+1};\F_p)$ can be interpreted in purely algebraic terms as the group homology of $B_n/Z(B_n)$ with trivial coefficients. Besides $H_*(B_n/Z(B_n);\F_p)$, in what follows 
		Similarly, the quotient of $grav(n)$ by the subspace $<x-(-1)^{\sigma}\sigma\cdot x\mid x\in grav(n), \sigma\in \Sigma_n>$ computes $sH_*^{\Sigma_n}(\M_{0,n+1};\F_p(\pm 1))$. So,
		
		\[
		H_*\left(\frac{grav(n)}{<x-(-1)^{\sigma}\sigma\cdot x>}\right)\cong sH_*^{\Sigma_n}(\M_{0,n+1};\F_p(\pm 1))
		\]   
	\end{oss}
	We are now ready to define what is an equivariant homology operation for a gravity algebra:
	\begin{defn}[Equivariant operations for even classes]
		Let $Q\in sH_*^{\Sigma_n}(\M_{0,n+1};\F_p)$ and let $q\in grav(n)$ be an element such that $[q]\in grav(n)_{\Sigma_n}$ is a representative for the class $sQ\in sH_*^{\Sigma_n}(\M_{0,n+1};\F_p)$. Since there is a shift of degree, $\abs{q}=\abs{Q}+1$. If $[a]\in H_*(A)$ is a even degree class, then it is not hard to verify that $q\otimes a^{\otimes n}$ is a cycle in $grav(n)\otimes_{\Sigma_n}A^{\otimes n}$. Then we define
		\[
		Q(a)\coloneqq\gamma_*(q\otimes a^{\otimes n})\in H_*(A)
		\]
		It is not hard to see that $Q(a)$ does not depend neither on the choice of $q$, nor on the choice of a representative cycle for $[a]$. So the definition is well posed. To sum up, any class $Q\in H_*^{\Sigma_n}(\M_{0,n+1};\F_p)$ gives rise to a homology operation (defined only for classes of even degrees)
		\[
		Q(-):H_{2m}(A)\to H_{2mn+\abs{Q}+1}(A)
		\]
	\end{defn}
	
	\begin{defn}[Equivariant operations for odd classes]
		Let $Q\in H_*^{\Sigma_n}(\M_{0,n+1};\F_p(\pm 1))$ and choose $q\in grav(n)$ be an element such that $[q]\in grav(n)/<x-(-1)^{\sigma}\sigma\cdot x\mid x\in grav(n), \sigma\in \Sigma_n>$ is a representative for the class $sQ\in sH_*^{\Sigma_n}(\M_{0,n+1});\F_p(\pm 1))$. Since there is a shift of degree, $\abs{q}=\abs{Q}+1$.  If $[a]\in H_*(A)$ is a odd degree class, then it is not hard to verify that $q\otimes a^{\otimes n}$ is a cycle in $grav(n)\otimes_{\Sigma_n}A^{\otimes n}$. Then we define
		\[
		Q(a)\coloneqq\gamma_*(q\otimes a^{\otimes n})\in H_*(A)
		\]
		It is not hard to see that $Q(a)$ does not depend neither on the the choice of $q$, nor on the choice of a representative cycle for $[a]$. So the definition is well posed. To sum up, any class $Q\in H_*^{\Sigma_n}(\M_{0,n+1};\F_p(\pm 1))$ gives rise to a homology operation (defined only for classes of odd degrees)
		\[
		Q(-):H_{2m+1}(A)\to H_{(2m+1)n+\abs{Q}+1}(A)
		\]
	\end{defn}
	To conclude, if one wants to understand homology operations for gravity algebras, the first step is to compute $H_*(B_n/Z(B_n);\F_p)$ and $H_*(B_n/Z(B_n);\F_p(\pm 1))$. This will be the main achievement of Section \ref{sec: operations for even classes} and Section \ref{sec: operations for odd degree classes}.
	
	\subsection{Homotopy models for $(\M_{0,n+1})_{\Sigma_n}$}\label{sec: quoziente omotopico come spazio classificante}
	
	As explained in the previous section, to understand the equivariant operations on gravity algebras one has to compute $H_*^{\Sigma_n}(\M_{0,n+1};\F_p)$. The first thing one might try to do is to study the Serre spectral sequence associated to the fibration
	\begin{equation}\label{eq: fibrazione difficile}
		\M_{0,n+1}\hookrightarrow (\M_{0,n+1})_{\Sigma_n}\to B\Sigma_n
	\end{equation}
	
	However the action of $\Sigma_n$ on the homology of $\M_{0,n+1}$ is not trivial and this complicates the whole computation. To overcome this problem the key observation is the following:
	\begin{lem}\label{lem:equivalenza omotopica tra quozienti omotopici}
		$(\mathcal{M}_{0,n+1})_{\Sigma_n}$ is homotopy equivalent to $C_n(\C)_{S^1}$.
	\end{lem}
	\begin{proof}
		Recall that $\M_{0,n+1}$ is $\Sigma_n$-homotopy equivalent to $F_n(\C)/S^1$, so the homotopy quotients $(\M_{0,n+1})_{\Sigma_n}$ and $(F_n(\C)/S^1)_{\Sigma_n}$ are homotopy equivalent. $S^1$ acts freely on $F_n(\C)$ so $(F_n(\C)/S^1)_{\Sigma_n}$ is homotopy equivalent to $(F_n(\C)_{S^1})_{\Sigma_n}$. The action of $\Sigma_n$ on $F_n(\C)$ and that of $S^1$ commute so we get a homotopy equivalence between  $(F_n(\C)_{S^1})_{\Sigma_n}$ and  $(F_n(\C)_{\Sigma_n})_{S^1}$, which is homotopy equivalent to $C_n(\C)_{S^1}$ since $\Sigma_n$ acts freely on $F_n(\C)$.
	\end{proof}
	Thanks to this Lemma we have an isomorphism between $H_*^{\Sigma_n}(\M_{0,n+1})$ and $H_*^{S^1}(C_n(\C))$. Section \ref{sec: operations for even classes} will be dedicated to the computation of $H_*^{S^1}(C_n(\C);\F_p)$ for any $n\in\N$ and any prime number $p$. This is obtained by the Serre spectral sequence associated to 
	\begin{equation}\label{eq:fibrazione facile}
		C_n(\C)\hookrightarrow C_n(\C)_{S^1}\to BS^1
	\end{equation}
	\begin{oss}
		The spectral sequence of the fibration \ref{eq:fibrazione facile} is much easier than the one associated to \ref{eq: fibrazione difficile}. First of all in this case we have trivial monodromy since $BS^1$ is simply connected. If we fix a field of coefficients $\F$ for (co)homology, the $E_2$ page of the spectral sequence is given by
		\[
		E^2_{p,q}=H_p(C_n(\C))\otimes H_q(BS^1)
		\]
		In addition, the differential $d^2$ of the $E^2$ page is well known in this case, as we will explain in the rest of this section.
	\end{oss}
	\begin{prop}\label{prop: sequenza spettrale di X per cerchio}
		Let $X$ be any topological space such that $H_i(X;\Z)$ is a finitely generated abelian group for each $i\in\N$. Consider $ S^1\times X$ with the natural action of $S^1$ by multiplication on the left and fix a field $\F$ of coefficients for (co)homology. Then the homological spectral sequence associated to the fibration $ S^1\times X\to (S^1\times X)_{S^1}\to BS^1$ has the following form:
		\begin{enumerate}
			\item $E^2=H_*(S^1)\otimes H_*(X)\otimes H_*(BS^1)$.
			\item Let $y_{2i}$ be a generator of $H_{2i}(BS^1;\F)$, $e_0$ a generator of $H_0(S^1)$ and $x\in H_*(X;\F)$. Then the differential $d^2$ of the second page is given by
			\[
			d^2(e_0\otimes x\otimes y_{2i})=\begin{cases} 0 \text{ if } i=0\\
				[S^1]\otimes x\otimes y_{2i-2} \text{ otherwise}
				
			\end{cases}
			\qquad d^2([S^1]\otimes x\otimes y_{2i})=0
			\]
			\item The spectral sequence degenerates at the third page, which is given by:
			
			\[
			E^3_{i,j}=\begin{cases}
				e_0\otimes H_j(X)\otimes y_0 \text{ if } i=0\\
				0 \text{ if } i>0
			\end{cases}
			\]
		\end{enumerate}

	\end{prop}
	\begin{proof}
		Point $(1)$ is clear, $(3)$ follows from $(2)$. So the only thing to prove is the statement of point $(2)$. Since $S^1$ acts freely on $S^1\times X$, the homotopy quotient $(S^1\times X)_{S^1}$ is homotopy equivalent to the strict quotient $(S^1\times X)/S^1=X$. The original fiber sequence can be rewritten as $S^1\times X\to X\to BS^1$, where the first map $p:S^1\times X\to X$ is the projection on the second factor. We prove the dual statement in order to exploit the multiplicativity of the cohomological spectral sequence. The second page looks as follows:
		\[
		E_2=\frac{\F[a]}{(a^2)}\otimes H^*(X;\F)\otimes \F[c]
		\]
		where $a$ is a generator of $H^1(S^1)$ and $c$ is a generator of $H^2(BS^1)$. The classes $y\in H^*(X)\subseteq E_2^{0,*}$ are infinite cycles because they belong to the image of $p^*:H^*(X)\to H^*(S^1\times X)$. This observation implies that $E_3=E_{\infty}$, because the only multiplicative generator which can have non zero differentials is $a$, which is a class in $E_2^{0,1}$, and therefore $d_n(a)=0$ for $n\geq 3$. Now we claim that $d_2(a)$ is a generator of $E_2^{2,0}=\F c$ : consider the projection $p:S^1\times X\to S^1$ on the first factor. This map is $S^1$-equivariant, so we get a map of fibrations
		\[
		\begin{tikzcd}
			&S^1\times X \arrow[r] \arrow[d,hook] & S^1 \arrow[d,hook]\\
			& X \arrow[d] \arrow[r] & \ast \arrow[d]\\
			& BS^1 \arrow[r] & BS^1
		\end{tikzcd}
		\]
		The claim now follows by comparing the spectral sequences of the right and left fibration.
	\end{proof}
	\begin{prop}\label{prop: differenziale dato da delta}
		Let $X$ be a topological space of finite type equipped with an $S^1$ action. Fix $\F$ a field of coefficients for (co)homology. Then the differential $d^2$ of the second page of the homological spectral sequence associated to $X\hookrightarrow X_{S^1}\to BS^1$ is given by 
		\[
		d^2(x\otimes y_{2i})=\begin{cases} 0 \text{ if } i=0\\
			\Delta(x)\otimes y_{2i-2} \text{ otherwise}
			
		\end{cases}
		\]
		where $y_{2i}$ is the generator of $H_{2i}(BS^1;\F)$.
	\end{prop}
	\begin{proof}
		Consider the map of fibrations
		\[
		\begin{tikzcd}
			& S^1\times X \arrow[r,"\theta"] \arrow[d,hook]& X\arrow[d,hook]\\
			& (S^1\times X)_{S^1} \arrow[r] \arrow[d]& X_{S^1}\arrow[d]\\
			&BS^1\arrow[r] &BS^1
		\end{tikzcd}
		\]
		The statement follows combining the definition of $\Delta$ and the formula for $d^2$ given in Proposition \ref{prop: sequenza spettrale di X per cerchio}.
	\end{proof}
	We end this Section observing that $(\M_{0,n+1})_{\Sigma_n}\simeq C_n(\C)_{S^1}$ is the classifying space for the quotient of the braid group by its center: let $B_n$ be the Braid group on $n$-strands and let us denote by $\sigma_i$ the $i$-th generator of $B_n$. It is well known that the center of $B_n$ is an infinite cyclic group generated by $\delta^2$, where $\delta\coloneqq\sigma_1(\sigma_2\sigma_1)(\sigma_3\sigma_2\sigma_1)\dots(\sigma_{n-1}\sigma_{n-2}\dots\sigma_1)$ is the following braid:

		\centering

		\tikzset{every picture/.style={line width=0.75pt}} %set default line width to 0.75pt        
		
		\begin{tikzpicture}[x=0.75pt,y=0.75pt,yscale=-1,xscale=1]
			%uncomment if require: \path (0,429); %set diagram left start at 0, and has height of 429
			
			%Straight Lines [id:da7631867355492767] 
			\draw    (83,70) -- (162.5,70) ;
			%Shape: Circle [id:dp27504299878891136] 
			\draw  [color={rgb, 255:red, 0; green, 0; blue, 0 }  ,draw opacity=1 ][fill={rgb, 255:red, 0; green, 0; blue, 0 }  ,fill opacity=1 ] (100.77,70) .. controls (100.77,68.77) and (101.77,67.77) .. (103,67.77) .. controls (104.23,67.77) and (105.23,68.77) .. (105.23,70) .. controls (105.23,71.23) and (104.23,72.23) .. (103,72.23) .. controls (101.77,72.23) and (100.77,71.23) .. (100.77,70) -- cycle ;
			%Shape: Circle [id:dp918969865412917] 
			\draw  [color={rgb, 255:red, 0; green, 0; blue, 0 }  ,draw opacity=1 ][fill={rgb, 255:red, 0; green, 0; blue, 0 }  ,fill opacity=1 ] (80.77,70) .. controls (80.77,68.77) and (81.77,67.77) .. (83,67.77) .. controls (84.23,67.77) and (85.23,68.77) .. (85.23,70) .. controls (85.23,71.23) and (84.23,72.23) .. (83,72.23) .. controls (81.77,72.23) and (80.77,71.23) .. (80.77,70) -- cycle ;
			%Shape: Circle [id:dp6072198378804974] 
			\draw  [color={rgb, 255:red, 0; green, 0; blue, 0 }  ,draw opacity=1 ][fill={rgb, 255:red, 0; green, 0; blue, 0 }  ,fill opacity=1 ] (120.52,70) .. controls (120.52,68.77) and (121.52,67.77) .. (122.75,67.77) .. controls (123.98,67.77) and (124.98,68.77) .. (124.98,70) .. controls (124.98,71.23) and (123.98,72.23) .. (122.75,72.23) .. controls (121.52,72.23) and (120.52,71.23) .. (120.52,70) -- cycle ;
			%Shape: Circle [id:dp3134999193912873] 
			\draw  [color={rgb, 255:red, 0; green, 0; blue, 0 }  ,draw opacity=1 ][fill={rgb, 255:red, 0; green, 0; blue, 0 }  ,fill opacity=1 ] (140.8,70.26) .. controls (140.8,69.03) and (141.8,68.03) .. (143.03,68.03) .. controls (144.26,68.03) and (145.26,69.03) .. (145.26,70.26) .. controls (145.26,71.49) and (144.26,72.49) .. (143.03,72.49) .. controls (141.8,72.49) and (140.8,71.49) .. (140.8,70.26) -- cycle ;
			%Shape: Circle [id:dp22704330988028754] 
			\draw  [color={rgb, 255:red, 0; green, 0; blue, 0 }  ,draw opacity=1 ][fill={rgb, 255:red, 0; green, 0; blue, 0 }  ,fill opacity=1 ] (160.27,70) .. controls (160.27,68.77) and (161.27,67.77) .. (162.5,67.77) .. controls (163.73,67.77) and (164.73,68.77) .. (164.73,70) .. controls (164.73,71.23) and (163.73,72.23) .. (162.5,72.23) .. controls (161.27,72.23) and (160.27,71.23) .. (160.27,70) -- cycle ;
			%Curve Lines [id:da9886822033187641] 
			\draw    (83,70) .. controls (83.5,92.4) and (162.5,95.77) .. (162.5,152.77) ;
			%Curve Lines [id:da8960918776184916] 
			\draw    (103,70) .. controls (102.5,78.6) and (99.5,79.6) .. (95.5,81.6) ;
			%Curve Lines [id:da10926564845510267] 
			\draw    (91,85) .. controls (87.5,86.4) and (84.5,86.8) .. (83.5,90.8) ;
			%Curve Lines [id:da01261339851973342] 
			\draw    (122.75,72.23) .. controls (122.25,80.83) and (111.5,86.4) .. (107.5,88.4) ;
			%Curve Lines [id:da7855372622358108] 
			\draw    (143.03,72.49) .. controls (142.53,81.09) and (127.5,93.4) .. (123.5,95.4) ;
			%Curve Lines [id:da4867959642740667] 
			\draw    (162.5,72.23) .. controls (162,80.83) and (140.5,101.4) .. (136.5,103.4) ;
			%Curve Lines [id:da5860080636410685] 
			\draw    (83.5,90.8) .. controls (79.5,102.8) and (139.5,105.4) .. (143.03,153.03) ;
			%Straight Lines [id:da17786971781198413] 
			\draw    (83,155) -- (162.5,155) ;
			%Shape: Circle [id:dp5448066627745958] 
			\draw  [color={rgb, 255:red, 0; green, 0; blue, 0 }  ,draw opacity=1 ][fill={rgb, 255:red, 0; green, 0; blue, 0 }  ,fill opacity=1 ] (100.77,155) .. controls (100.77,153.77) and (101.77,152.77) .. (103,152.77) .. controls (104.23,152.77) and (105.23,153.77) .. (105.23,155) .. controls (105.23,156.23) and (104.23,157.23) .. (103,157.23) .. controls (101.77,157.23) and (100.77,156.23) .. (100.77,155) -- cycle ;
			%Shape: Circle [id:dp3900384628967788] 
			\draw  [color={rgb, 255:red, 0; green, 0; blue, 0 }  ,draw opacity=1 ][fill={rgb, 255:red, 0; green, 0; blue, 0 }  ,fill opacity=1 ] (80.77,155) .. controls (80.77,153.77) and (81.77,152.77) .. (83,152.77) .. controls (84.23,152.77) and (85.23,153.77) .. (85.23,155) .. controls (85.23,156.23) and (84.23,157.23) .. (83,157.23) .. controls (81.77,157.23) and (80.77,156.23) .. (80.77,155) -- cycle ;
			%Shape: Circle [id:dp05980275940728319] 
			\draw  [color={rgb, 255:red, 0; green, 0; blue, 0 }  ,draw opacity=1 ][fill={rgb, 255:red, 0; green, 0; blue, 0 }  ,fill opacity=1 ] (120.52,155) .. controls (120.52,153.77) and (121.52,152.77) .. (122.75,152.77) .. controls (123.98,152.77) and (124.98,153.77) .. (124.98,155) .. controls (124.98,156.23) and (123.98,157.23) .. (122.75,157.23) .. controls (121.52,157.23) and (120.52,156.23) .. (120.52,155) -- cycle ;
			%Shape: Circle [id:dp3030188237392386] 
			\draw  [color={rgb, 255:red, 0; green, 0; blue, 0 }  ,draw opacity=1 ][fill={rgb, 255:red, 0; green, 0; blue, 0 }  ,fill opacity=1 ] (140.8,155.26) .. controls (140.8,154.03) and (141.8,153.03) .. (143.03,153.03) .. controls (144.26,153.03) and (145.26,154.03) .. (145.26,155.26) .. controls (145.26,156.49) and (144.26,157.49) .. (143.03,157.49) .. controls (141.8,157.49) and (140.8,156.49) .. (140.8,155.26) -- cycle ;
			%Shape: Circle [id:dp8354494083864512] 
			\draw  [color={rgb, 255:red, 0; green, 0; blue, 0 }  ,draw opacity=1 ][fill={rgb, 255:red, 0; green, 0; blue, 0 }  ,fill opacity=1 ] (160.27,155) .. controls (160.27,153.77) and (161.27,152.77) .. (162.5,152.77) .. controls (163.73,152.77) and (164.73,153.77) .. (164.73,155) .. controls (164.73,156.23) and (163.73,157.23) .. (162.5,157.23) .. controls (161.27,157.23) and (160.27,156.23) .. (160.27,155) -- cycle ;
			%Curve Lines [id:da06468750656042532] 
			\draw    (103.75,91.23) .. controls (100.5,92.4) and (96.5,94.2) .. (91.5,96.8) ;
			%Curve Lines [id:da8808036861470907] 
			\draw    (117.75,98.23) .. controls (114.5,99.4) and (109.5,102.8) .. (106.5,104.4) ;
			%Curve Lines [id:da8958721850756366] 
			\draw    (132.75,106.23) .. controls (128.5,108.8) and (125.5,109.8) .. (120.5,112.4) ;
			%Curve Lines [id:da8225153218194972] 
			\draw    (87,99) .. controls (82.5,102.8) and (83.5,103.8) .. (83.5,106.8) ;
			%Curve Lines [id:da939035173889607] 
			\draw    (83.5,106.8) .. controls (83.5,117.8) and (119.22,107.37) .. (122.75,155) ;
			%Curve Lines [id:da11897409700529371] 
			\draw    (102.75,107.23) .. controls (100.5,108.8) and (97.5,109.8) .. (93.5,112.4) ;
			%Curve Lines [id:da11664070211745559] 
			\draw    (116.75,114.23) .. controls (114.5,115.8) and (110.5,117.8) .. (107.5,118.8) ;
			%Curve Lines [id:da29884346580003895] 
			\draw    (90,115) .. controls (86.5,116.8) and (84.5,117.8) .. (86.5,123.8) ;
			%Curve Lines [id:da5824598040918343] 
			\draw    (86.5,123.8) .. controls (91.5,129.8) and (103.5,127.8) .. (103,155) ;
			%Curve Lines [id:da339083960107476] 
			\draw    (103.75,121.23) .. controls (100.5,122.4) and (97.5,124.2) .. (94.5,125.8) ;
			%Curve Lines [id:da6284850065279874] 
			\draw    (90,129) .. controls (83.5,132.8) and (83.5,142.8) .. (83,155) ;
			
			% Text Node
			\draw (44,106) node [anchor=north west][inner sep=0.75pt]   [align=left] {$\displaystyle \delta =$};

		\end{tikzpicture}
	\flushleft
	\begin{oss}
		The center of the Pure braid group $PB_n$ is an infinite cyclic group as well, generated by $\delta^2$.
	\end{oss}
	\begin{prop}\label{prop: quizienti omotopici come spazi classificanti}
		The homotopy quotient $C_n(\C)_{S^1}$ is the classifying space for the group $B_n/Z(B_n)$. Similarly, $F_n(\C)_{S^1}$ is the classifying space for $PB_n/Z(PB_n)$.
	\end{prop}
	\begin{proof}
		Consider the long exact sequence for homotopy groups associated to the fibration $S^1\hookrightarrow ES^1\times C_n(\C)\to C_n(\C)_{S^1}$. To get the result just observe that the map $i_*:\pi_1(S^1)\to\pi_1(ES^1\times C_n(\C))\cong\pi_1(C_n(\C))$ induced by the inclusion of a fiber sends the generator of $\pi_1(S^1)$ to $\delta^2$. The case of the ordered configurations is completely analogous.
	\end{proof}
	We can summarize these observations in the following table:
	\[
	\begin{tabular}{lc}
		\toprule
		Group & Models for the classifying space \\
		\midrule
		\medskip
		$PB_n$ & $F_n(\C)$ \\
		\midrule
		\medskip
		$B_n$ & $C_n(\C)$ \\
		\midrule
		\medskip
		$PB_n/Z(PB_n)$ & $F_n(\C)_{S^1}$ \quad$ \mathcal{M}_{0,n+1}$ \\
		\midrule
		\medskip
		$B_n/Z(B_n)$ & $C_n(\C)_{S^1}$\quad $ (\mathcal{M}_{0,n+1})_{\Sigma_n}$\\
		\bottomrule
	\end{tabular}
	\]
	
	\section{Operations for even degree classes}\label{sec: operations for even classes}
	
	In this section we compute $H_*^{\Sigma_n}(\mathcal{M}_{0,n+1};\mathbb{F}_p)$ for any $n\in\mathbb{N}$ and any prime number $p$. As explained in Section \ref{sec: quoziente omotopico come spazio classificante} this computation can be done by looking at the Serre spectral sequence associated to the fibration
	\begin{equation*}
		\begin{tikzcd}
			& C_n(\C)\arrow[r,hook] & C_n(\C)_{S^1} \arrow[r] & BS^1   
		\end{tikzcd}
	\end{equation*}
	instead of
	\[
	\begin{tikzcd}
		& \M_{0,n+1} \arrow[r,hook] & (\M_{0,n+1})_{\Sigma_n} \arrow[r] & B\Sigma_n   
	\end{tikzcd}
	\]
	Before going into the details of the computation, let us review some preliminary results.
	\subsection{Preliminares}\label{sec: preliminari}
	We start by reviewing the basics of equivariant cohomology.  We refer to \cite{TomDieck} and \cite{BredonGroups} for further details. 
	\subsubsection{Equivariant cohomology} let $G=\Z/n$ or $S^1$ and $M$ be an abelian group which we use as coefficients for (co)homology. We also suppose that $X$ is a finite dimensional $G$-complex of finite orbit type. Let $c\in H^2(BG;\Z)$ be a generator, and consider the multiplicative subset $S\coloneqq\{1,c,c^2,\dots,\}$. Consider the following subspace:
	\[
	FX\coloneqq\{x\in X\mid \Tilde{H}^*(BG_x;M)\neq 0\}
	\]
	where $G_x$ denotes the stabilizer of a point $x\in X$. A crucial result is the so called \emph{Localization Theorem}:
	\begin{thm}[\cite{TomDieck}, p.198]\label{thm:localization theorem}
		The inclusion $i:FX\to X$ induces an isomorphism
		\[
		i^*:S^{-1}H^*_G(X;M)\to S^{-1}H^*_G(FX;M)
		\]
		where $S^{-1}H^*_G(X;M)$ is the localization of the $H^*(BG;M)$-module $H^*_G(X;M)$ to the subset $S$.
	\end{thm}
	A consequence of this Theorem is the following:
	\begin{thm}[\cite{TomDieck}, p.199]\label{thm:isomorfismo coomologia equivariante tra spazio e i suoi punti fissi}
		Suppose $H^i(X;M)=0$ for $i>n$. Then the inclusion $i:FX\to X$ induces an isomorphism
		\[
		H^i_G(X;M)\to H^i_G(FX;M)
		\]
		for any $i>n-dim(G)$ and an epimorphism for $i=n-dim(G)$.
	\end{thm}
	Now let us restrict to the case $G=\Z/p$, with $p$ a prime number. In the following we will use $\F_p$ coefficients for the $\Z/p$-equivariant (co)homology of $X$.
	
	\begin{thm}[\cite{TomDieck} p. 200] \label{thm:tom Dieck degeneration theorem}
		Suppose $dim_{\F_p}\bigoplus_{k\in\N}H^k(X)$ is finite and $H^k(X)=0$ for any $k>n$. Then we have
		\begin{equation}\label{eq: diguguaglianza equivariante}
			dim_{\F_p}\bigoplus_{k\in\N}H^k(X^{\Z/p})\leq dim_{\F_p}\bigoplus_{k\in\N}H^k(X)   
		\end{equation}

		Moreover, the following assertions are equivalent:
		\begin{enumerate}
			\item Equality holds in (\ref{eq: diguguaglianza equivariante}).
			\item The map induced by the inclusion $i^*:H_{\Z/p}^*(X)\to H^*(X)$ is surjective.
			\item $dim_{\F_p}H_{\Z/p}^k(X)=dim_{\F_p}\bigoplus_{i\in\N}H^i(X)$ for $k>n$.
			\item $\Z/p$ acts trivially on $H^*(X)$ and the Serre spectral sequence of $X\hookrightarrow X_{\Z/p}\to B\Z/p$ degenerates at the $E_2$ page.
		\end{enumerate}
		
	\end{thm}
	\begin{oss}
		The previous Theorem holds as well if we replace $\Z/p$ with $S^1$ and take $\Q$ as field of coefficients. 
	\end{oss}
	
	\subsubsection{Labelled configuration spaces} \label{ref: labelled configuration spaces}
	\begin{defn}[B\"odigheimer, \cite{Bodigheimer}]
		Let $Y$ be a topological space and $(X,\ast)$ be a based CW-complex, not necessarily connected. The space of configurations in $Y$ with labels in $X$ is defined as 
		\[
		C(Y;X)\coloneqq\bigsqcup_{n\in\N}F_n(M)\times_{\Sigma_n} X^n/\sim
		\]
		where $(p_1,\dots,p_n;x_1,\dots,x_n)\sim (p_1,\dots,\hat{p}_i,\dots,p_n;x_1,\dots,\hat{x}_i,\dots,x_n)$ if $x_i=\ast$.
	\end{defn}
	\begin{es}
		$C(\C;S^0)$ is just the disjoint union $\bigsqcup_{n\in\N}C_n(\C)$. To ease the notation we sometimes abbreviate $C(\C;S^0)$ by $C(\C)$.
	\end{es}
	The homology of $C(\R^n;X)$ is known, thanks to the work of Cohen \cite{Cohen}. The idea is the following: $C(\R^n;X)$ is homotopy equivalent to the free $\mathcal{D}_n$-algebra on $X$, and Cohen showed that $H_*(C(\R^n;X);\F_p)$ can be described as a functor of $H_*(X;\F_p)$. For the purpose of this work it is enough to recall the results in the case $n=2$. 
	
	\begin{defn}
		Let $p$ be a prime number. Fix a basis $\mathcal{B}$ of $H_*(X;\F_p)/[\ast]$, where $[\ast]\in H_0(X;\F_p)$ is the class of the basepoint. We define a \textbf{basic bracket of weight k} inductively as follows:
		\begin{itemize}
			\item A basic bracket of weight $1$ is just an element $a\in\mathcal{B}$. Its degree is by definition the homological degree of $a$. Observe that any class of $H_*(X;\F_p)$ can be seen as a class in $H_*(C(\C;X);\F_p)$.
			\item By induction assume that the basic brackets of weight $j$ have been defined and equipped with a total ordering compatible with weight for $j<k$. Then a basic bracket of weight $k$ is a homology class $[a,b]\in H_*(C(\C;X);\F_p)$, where $[-,-]$ is the Browder bracket and $a,b$ are basic brackets such that:
			\begin{enumerate}
				\item $weight(a)+weight(b)=k$.
				\item $a<b$ and if $b=[c,d]$ then $c\leq a$.
			\end{enumerate}
			The degree of $[a,b]$ is by definition $deg(a)+deg(b)+1$.
			
		\end{itemize}
		In the case $p\neq2$ we also include as basic brackets classes of the form $[a,a]$, where $a$ is a basic bracket of even degree. 
	\end{defn}
	\begin{thm}[Cohen, \cite{Cohen}]\label{thm: omologia spazi di configurazione etichettati}
		Let $p$ be any prime. We will denote by $Q:H_q(C(\C;X);\F_p)\to H_{pq+p-1}(C(\C;X);\F_p)$ the first Dyer-Lashof operation (when $p$ is odd it acts only on classes of odd degree $q$). Then $H_*(C(\C;X);\F_p)$ has the following form:
		\begin{description}
			\item[$p=2$:] $H_*(C(\C;X);\F_2)$ is the free graded commutative algebra on classes $Q^i(x)$, where $Q^i$ denotes the $i$-th iteration of $Q$ and $x$ is a basic bracket.
			\item[$p\neq 2$:] $H_*(C(\C;X);\F_p)$ is the free graded commutative algebra on classes $Q^i(x)$ and $\beta Q^i(x)$, where $\beta$ is the Bockstein operator, $Q^i$ denotes the $i$-th iteration of $Q$ and $x$ is a basic bracket of odd degree.
		\end{description}
	\end{thm}
	\begin{corollario}
		Let $C(\C)\coloneqq\bigsqcup_{n\in\N}C_n(\C)$ be the disjoint union of all unordered configuration spaces of points in the complex plane. Then if $p$ is an odd prime
		\[
		H_*(C(\C);\F_p)=\F_p[\iota,\beta Q[\iota,\iota],\beta Q^2[\iota,\iota],\dots]\otimes\Lambda[[\iota,\iota],Q[\iota,\iota],Q^2[\iota,\iota],\dots]
		\]
		where $\iota$ is the generator of $H_0(C_1(\C))$. When $p=2$ we have
		\[
		H_*(C(\C);\F_2)=\F_2[\iota,Q\iota,Q^2\iota,\dots]
		\]
		
	\end{corollario}
	\begin{oss}\label{oss:notazione per omologia}
		In what follows we will adopt the following notation:
		\begin{align*}
			&u\coloneqq[\iota,\iota]\\
			&\beta_i\coloneqq \beta Q^i[\iota,\iota]\\
			&\alpha_i\coloneqq Q^i[\iota,\iota]  
		\end{align*}
		Moreover, the following table will be useful in the next section:
		\[
		\begin{tabular}{ccc}
			
			\toprule
			Homology class  & Number of points & Degree\\
			\midrule
			$\iota$ & $1$ & $0$ \\
			$u$ & $2$ & $1$ \\
			$\alpha_i$ & $2p^i$ & $2p^i-1$ \\
			$\beta_i$ & $2p^i$ & $2p^i-2$\\
			\bottomrule
		\end{tabular}
		\]
		
	\end{oss}
	
	\subsection{Computation of $H_*^{\Z/p}(C_n(\C);\F_p)$ when $n=0,1$ mod $p$}\label{sec: omologia equivariante spazio configurazioni rispetto a gruppo ciclico}
	
	In this section we compute $H_*^{\Z/p}(C_n(\C);\F_p)$ when $p$ is a prime that divides $n$ or $n-1$. Here we are considering $\Z/p$ as the subgroup of $p$-th roots of unity inside $S^1$, so its generator acts on $C_n(\C)$ by the rotation of $2\pi/p$. Some of the statements of this section are stated with the assumption that $p$ is an odd prime, but similar statements holds for $p=2$ with minor modifications. A different proof for the case $p=2$ will be included in Section \ref{sec: operations for odd degree classes} (see Corollary \ref{thm: collasso sequenza spettrale nel caso p=2}).
	
	\begin{thm}\label{thm:degeneration of a spectral sequence}
		Let $p$ be a prime, $n\in\N$ such that $p|n$ or $p|n-1$. Then 
		\[
		H_*^{\Z/p}(C_n(\C);\F_p)\cong H_*(C_n(\C);\F_p)\otimes H_*(B\Z/p;\F_p)
		\]
	\end{thm}
	\begin{proof}
		Consider the homological spectral sequence associated to the fibration $C_n(\C)\hookrightarrow C_n(\C)_{\Z/p}\to B\Z/p$. Since $\Z/p$ acts by rotations on $C_n(\C)$, the monodromy action is trivial. Therefore
		\[
		E^2_{p,q}\cong H_p(C_n(\C);\F_p)\otimes H_q(B\Z/p;\F_p)
		\]
		We will see that
		\begin{equation}\label{eq: configuration spaces}
			dim_{\F_p}\bigoplus_{k\in\N}H^k(C_n(\C)^{\Z/p})= dim_{\F_p}\bigoplus_{k\in\N}H^k(C_n(\C))  
		\end{equation}
		So the result follows applying Theorem \ref{thm:tom Dieck degeneration theorem}.
	\end{proof}
	
	Now we focus in proving the equality \ref{eq: configuration spaces}. For the moment we restrict to the case $n=pq$, then we will extend the result to the case $n=pq+1$.
	\begin{lem}\label{lem: Z/p punti fissi}
		Let $n=pq$ or $n=pq+1$. Then the fixed points $C_n(\C)^{\Z/p}$ are homeomorphic to $C_q(\C^*)$.
	\end{lem}
	\begin{proof}
		Let us prove the statement when $n=pq$, the other case in similar. Let us denote by $\zeta\coloneqq e^{i2\pi/p}$ the generator of $\Z/p$. Consider the quotient space
		\[
		H\coloneqq\{z\in\C^*\mid arg(z)\in[0,2\pi/p]\}/\sim
		\]
		where $\sim$ identifies a point $z\in\{z\in\C^*\mid arg(z)=0\}$ with $\zeta z$. So $H$ is homeomorphic to $\C^*$. Now observe that any configuration in $C_{n}(\C)^{\Z/p}$ is of the form $\{z_1,\zeta z_1,\dots\zeta^{p-1}z_1,\dots, z_q,\zeta z_q,\dots\zeta^{p-1}z_q\}$, where $z_1,\dots,z_q$ are distinct points in $\{z\in\C^*\mid arg(z)\in[0,2\pi/p)\}$. The association $\{z_1,\zeta z_1,\dots\zeta^{p-1}z_1,\dots, z_q,\zeta z_q,\dots\zeta^{p-1}z_q\}\mapsto\{z_1,\dots,z_q\}$ defines a continuous map 
		\[
		f: C_{n}(\C)^{\Z/p}\to C_q(H)
		\]
		Conversely, if we have a configuration $\{z_1,\dots,z_q\}\in  C_q(H)$, we can produce a configuration of $C_n(\C)^{\Z/p}$ by taking the $\Z/p$-orbits of every point. More precisely, the association $\{z_1,\dots,z_q\}\mapsto\{z_1,\zeta z_1,\dots\zeta^{p-1}z_1,\dots, z_q,\zeta z_q,\dots\zeta^{p-1}z_q\}$ defines a continuous function $C_q(H)\to C_n(\C)^{\Z/p}$, which is the inverse of $f$. See Figure \ref{fig:homeomorfismo dello spazio dei punti fissi} for a pictorial description of $f$.
	\end{proof}
	\begin{figure}
		\centering

		\tikzset{every picture/.style={line width=0.75pt}} %set default line width to 0.75pt        
		
		\begin{tikzpicture}[x=0.75pt,y=0.75pt,yscale=-1,xscale=1]
			%uncomment if require: \path (0,627); %set diagram left start at 0, and has height of 627
			
			%Straight Lines [id:da3071243754336681] 
			\draw    (140.15,163.71) -- (234.79,163.71) ;
			%Straight Lines [id:da9906246270321968] 
			\draw    (140.15,163.71) -- (169.4,73.71) ;
			%Shape: Ellipse [id:dp40596779385040493] 
			\draw  [fill={rgb, 255:red, 0; green, 0; blue, 0 }  ,fill opacity=1 ] (174.27,144) .. controls (174.27,143.25) and (174.88,142.64) .. (175.63,142.64) .. controls (176.39,142.64) and (177,143.25) .. (177,144) .. controls (177,144.75) and (176.39,145.36) .. (175.63,145.36) .. controls (174.88,145.36) and (174.27,144.75) .. (174.27,144) -- cycle ;
			%Shape: Circle [id:dp21243961222397467] 
			\draw  [fill={rgb, 255:red, 0; green, 0; blue, 0 }  ,fill opacity=1 ] (194.01,148.23) .. controls (194.01,147.48) and (194.62,146.87) .. (195.37,146.87) .. controls (196.12,146.87) and (196.73,147.48) .. (196.73,148.23) .. controls (196.73,148.98) and (196.12,149.59) .. (195.37,149.59) .. controls (194.62,149.59) and (194.01,148.98) .. (194.01,148.23) -- cycle ;
			%Shape: Ellipse [id:dp4524449710807179] 
			\draw  [fill={rgb, 255:red, 0; green, 0; blue, 0 }  ,fill opacity=1 ] (177.09,111.58) .. controls (177.09,110.83) and (177.7,110.22) .. (178.45,110.22) .. controls (179.2,110.22) and (179.81,110.83) .. (179.81,111.58) .. controls (179.81,112.33) and (179.2,112.94) .. (178.45,112.94) .. controls (177.7,112.94) and (177.09,112.33) .. (177.09,111.58) -- cycle ;
			%Straight Lines [id:da511367916758239] 
			\draw    (139.99,163.6) -- (169.03,253.68) ;
			%Shape: Ellipse [id:dp270160442995391] 
			\draw  [fill={rgb, 255:red, 0; green, 0; blue, 0 }  ,fill opacity=1 ] (169.22,190.03) .. controls (169.93,189.8) and (170.7,190.19) .. (170.93,190.91) .. controls (171.16,191.62) and (170.77,192.39) .. (170.05,192.62) .. controls (169.34,192.85) and (168.57,192.46) .. (168.34,191.74) .. controls (168.11,191.03) and (168.5,190.26) .. (169.22,190.03) -- cycle ;
			%Shape: Ellipse [id:dp8909265302680107] 
			\draw  [fill={rgb, 255:red, 0; green, 0; blue, 0 }  ,fill opacity=1 ] (171.25,210.11) .. controls (171.97,209.88) and (172.73,210.27) .. (172.96,210.99) .. controls (173.19,211.7) and (172.8,212.47) .. (172.09,212.7) .. controls (171.37,212.93) and (170.6,212.54) .. (170.37,211.82) .. controls (170.14,211.11) and (170.53,210.34) .. (171.25,210.11) -- cycle ;
			%Shape: Ellipse [id:dp9751727993684856] 
			\draw  [fill={rgb, 255:red, 0; green, 0; blue, 0 }  ,fill opacity=1 ] (200.94,182.76) .. controls (201.66,182.53) and (202.42,182.92) .. (202.66,183.64) .. controls (202.89,184.36) and (202.49,185.12) .. (201.78,185.35) .. controls (201.06,185.59) and (200.29,185.19) .. (200.06,184.48) .. controls (199.83,183.76) and (200.23,182.99) .. (200.94,182.76) -- cycle ;
			%Straight Lines [id:da49644956147150876] 
			\draw    (140.11,163.96) -- (63,218.83) ;
			%Shape: Ellipse [id:dp13809628010070107] 
			\draw  [fill={rgb, 255:red, 0; green, 0; blue, 0 }  ,fill opacity=1 ] (123.74,199.8) .. controls (124.18,200.41) and (124.03,201.26) .. (123.42,201.7) .. controls (122.81,202.13) and (121.96,201.99) .. (121.52,201.38) .. controls (121.09,200.77) and (121.23,199.92) .. (121.84,199.48) .. controls (122.45,199.04) and (123.3,199.19) .. (123.74,199.8) -- cycle ;
			%Shape: Ellipse [id:dp9280965467285764] 
			\draw  [fill={rgb, 255:red, 0; green, 0; blue, 0 }  ,fill opacity=1 ] (105.21,207.8) .. controls (105.65,208.41) and (105.5,209.26) .. (104.89,209.7) .. controls (104.28,210.13) and (103.43,209.99) .. (102.99,209.38) .. controls (102.55,208.76) and (102.7,207.91) .. (103.31,207.48) .. controls (103.92,207.04) and (104.77,207.19) .. (105.21,207.8) -- cycle ;
			%Shape: Ellipse [id:dp967326548618374] 
			\draw  [fill={rgb, 255:red, 0; green, 0; blue, 0 }  ,fill opacity=1 ] (140.24,227.85) .. controls (140.68,228.46) and (140.54,229.31) .. (139.92,229.75) .. controls (139.31,230.18) and (138.46,230.04) .. (138.02,229.43) .. controls (137.59,228.82) and (137.73,227.97) .. (138.34,227.53) .. controls (138.96,227.09) and (139.81,227.24) .. (140.24,227.85) -- cycle ;
			%Straight Lines [id:da8037691665890196] 
			\draw    (139.99,163.6) -- (63.42,107.98) ;
			%Shape: Ellipse [id:dp8206553567860075] 
			\draw  [fill={rgb, 255:red, 0; green, 0; blue, 0 }  ,fill opacity=1 ] (131.62,125.59) .. controls (130.91,125.36) and (130.51,124.59) .. (130.75,123.88) .. controls (130.98,123.16) and (131.75,122.77) .. (132.46,123) .. controls (133.18,123.24) and (133.57,124) .. (133.34,124.72) .. controls (133.1,125.43) and (132.34,125.83) .. (131.62,125.59) -- cycle ;
			%Shape: Ellipse [id:dp2738784969719863] 
			\draw  [fill={rgb, 255:red, 0; green, 0; blue, 0 }  ,fill opacity=1 ] (141.74,108.13) .. controls (141.03,107.9) and (140.63,107.13) .. (140.87,106.42) .. controls (141.1,105.7) and (141.87,105.31) .. (142.58,105.54) .. controls (143.3,105.77) and (143.69,106.54) .. (143.46,107.26) .. controls (143.22,107.97) and (142.46,108.36) .. (141.74,108.13) -- cycle ;
			%Shape: Ellipse [id:dp4505250339267324] 
			\draw  [fill={rgb, 255:red, 0; green, 0; blue, 0 }  ,fill opacity=1 ] (101.66,112.89) .. controls (100.94,112.66) and (100.55,111.89) .. (100.78,111.18) .. controls (101.01,110.46) and (101.78,110.07) .. (102.5,110.3) .. controls (103.21,110.54) and (103.6,111.3) .. (103.37,112.02) .. controls (103.14,112.73) and (102.37,113.13) .. (101.66,112.89) -- cycle ;
			%Shape: Ellipse [id:dp7403207305089852] 
			\draw  [fill={rgb, 255:red, 0; green, 0; blue, 0 }  ,fill opacity=1 ] (100.82,159.88) .. controls (100.37,160.48) and (99.52,160.61) .. (98.91,160.17) .. controls (98.31,159.73) and (98.17,158.88) .. (98.62,158.27) .. controls (99.06,157.66) and (99.91,157.53) .. (100.52,157.97) .. controls (101.13,158.42) and (101.26,159.27) .. (100.82,159.88) -- cycle ;
			%Shape: Ellipse [id:dp6888256439787033] 
			\draw  [fill={rgb, 255:red, 0; green, 0; blue, 0 }  ,fill opacity=1 ] (87.38,144.81) .. controls (86.94,145.42) and (86.09,145.55) .. (85.48,145.11) .. controls (84.87,144.67) and (84.74,143.81) .. (85.18,143.21) .. controls (85.63,142.6) and (86.48,142.47) .. (87.09,142.91) .. controls (87.69,143.35) and (87.83,144.21) .. (87.38,144.81) -- cycle ;
			%Shape: Ellipse [id:dp6133484322766214] 
			\draw  [fill={rgb, 255:red, 0; green, 0; blue, 0 }  ,fill opacity=1 ] (79.41,184.38) .. controls (78.96,184.99) and (78.11,185.12) .. (77.5,184.68) .. controls (76.9,184.24) and (76.76,183.38) .. (77.21,182.78) .. controls (77.65,182.17) and (78.5,182.04) .. (79.11,182.48) .. controls (79.72,182.93) and (79.85,183.78) .. (79.41,184.38) -- cycle ;
			%Straight Lines [id:da3991157308650157] 
			\draw    (264,166) -- (330.73,166) ;
			\draw [shift={(332.73,166)}, rotate = 180] [color={rgb, 255:red, 0; green, 0; blue, 0 }  ][line width=0.75]    (10.93,-3.29) .. controls (6.95,-1.4) and (3.31,-0.3) .. (0,0) .. controls (3.31,0.3) and (6.95,1.4) .. (10.93,3.29)   ;
			%Rounded Rect [id:dp39377797578416573] 
			\draw  [dash pattern={on 0.84pt off 2.51pt}] (39.73,102.88) .. controls (39.73,80.68) and (57.73,62.68) .. (79.93,62.68) -- (210.53,62.68) .. controls (232.73,62.68) and (250.73,80.68) .. (250.73,102.88) -- (250.73,223.48) .. controls (250.73,245.68) and (232.73,263.68) .. (210.53,263.68) -- (79.93,263.68) .. controls (57.73,263.68) and (39.73,245.68) .. (39.73,223.48) -- cycle ;
			%Straight Lines [id:da4873151783837908] 
			\draw    (360.46,206.41) -- (494.73,206.41) ;
			%Straight Lines [id:da13964024258913077] 
			\draw    (360.46,206.41) -- (401.95,78.71) ;
			%Shape: Circle [id:dp7271811837988691] 
			\draw  [fill={rgb, 255:red, 0; green, 0; blue, 0 }  ,fill opacity=1 ] (408.87,178.44) .. controls (408.87,177.37) and (409.73,176.51) .. (410.8,176.51) .. controls (411.87,176.51) and (412.73,177.37) .. (412.73,178.44) .. controls (412.73,179.5) and (411.87,180.37) .. (410.8,180.37) .. controls (409.73,180.37) and (408.87,179.5) .. (408.87,178.44) -- cycle ;
			%Shape: Circle [id:dp41261600712916136] 
			\draw  [fill={rgb, 255:red, 0; green, 0; blue, 0 }  ,fill opacity=1 ] (436.87,184.44) .. controls (436.87,183.37) and (437.73,182.51) .. (438.8,182.51) .. controls (439.87,182.51) and (440.73,183.37) .. (440.73,184.44) .. controls (440.73,185.5) and (439.87,186.37) .. (438.8,186.37) .. controls (437.73,186.37) and (436.87,185.5) .. (436.87,184.44) -- cycle ;
			%Shape: Circle [id:dp7805544232755597] 
			\draw  [fill={rgb, 255:red, 0; green, 0; blue, 0 }  ,fill opacity=1 ] (412.87,132.44) .. controls (412.87,131.37) and (413.73,130.51) .. (414.8,130.51) .. controls (415.87,130.51) and (416.73,131.37) .. (416.73,132.44) .. controls (416.73,133.5) and (415.87,134.37) .. (414.8,134.37) .. controls (413.73,134.37) and (412.87,133.5) .. (412.87,132.44) -- cycle ;
			%Shape: Circle [id:dp9700696787098535] 
			\draw  [fill={rgb, 255:red, 255; green, 255; blue, 255 }  ,fill opacity=1 ] (357.37,206.41) .. controls (357.37,204.7) and (358.76,203.32) .. (360.46,203.32) .. controls (362.17,203.32) and (363.55,204.7) .. (363.55,206.41) .. controls (363.55,208.11) and (362.17,209.49) .. (360.46,209.49) .. controls (358.76,209.49) and (357.37,208.11) .. (357.37,206.41) -- cycle ;
			%Straight Lines [id:da5457973847496012] 
			\draw    (494.73,206.41) -- (429.6,206.41) ;
			\draw [shift={(427.6,206.41)}, rotate = 360] [color={rgb, 255:red, 0; green, 0; blue, 0 }  ][line width=0.75]    (10.93,-3.29) .. controls (6.95,-1.4) and (3.31,-0.3) .. (0,0) .. controls (3.31,0.3) and (6.95,1.4) .. (10.93,3.29)   ;
			%Straight Lines [id:da458713063844854] 
			\draw    (401.95,78.71) -- (381.83,140.65) ;
			\draw [shift={(381.21,142.56)}, rotate = 288] [color={rgb, 255:red, 0; green, 0; blue, 0 }  ][line width=0.75]    (10.93,-3.29) .. controls (6.95,-1.4) and (3.31,-0.3) .. (0,0) .. controls (3.31,0.3) and (6.95,1.4) .. (10.93,3.29)   ;
			
			% Text Node
			\draw (292,143) node [anchor=north west][inner sep=0.75pt]   [align=left] {$\displaystyle f$};

		\end{tikzpicture}
		
		\caption{This picture shows how the homeomorphism $f:C_{15}(\C)^{\Z/5}\to C_3(\C^*)$ works. }
		\label{fig:homeomorfismo dello spazio dei punti fissi}
	\end{figure}
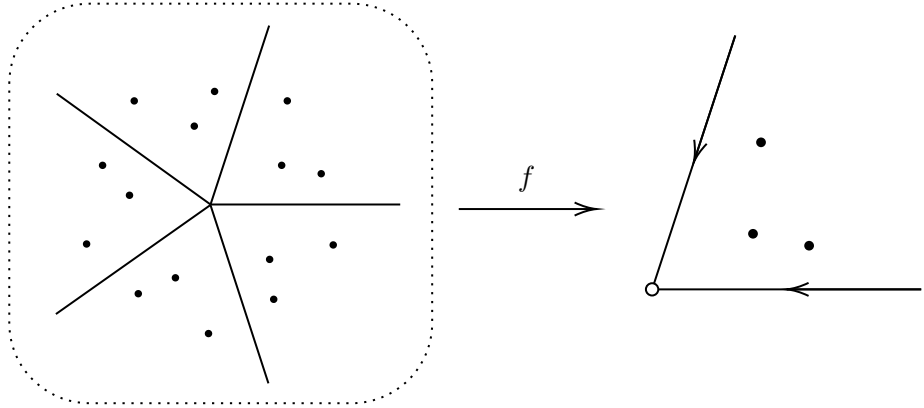
	\begin{oss}
		$C_q(\C^*)$ is homotopy equivalent to the configuration space of $q$ black particles and one white particle in the plane. Therefore $H_*(C_q(\C^*);\F_p)$ will be the subspace of $H_*(C(\C,S^0\vee S^0);\F_p)$ spanned by those classes that involve only $q$ black particles and one white particle. Let us denote by $a$ (resp $b$) the class in $H_0(S^0\vee S^0)$ which represent a black particle (resp. a white particle). Then $H_*(C(\C,S^0\vee S^0);\F_p)$ can be computed using Theorem \ref{thm: omologia spazi di configurazione etichettati}. The Lemma stated below just identifies explicitly $H_*(C_q(\C^*);\F_p)$ as a subspace of $H_*(C(\C,S^0\vee S^0);\F_p)$. 
	\end{oss}
	
	\begin{lem}
		Let $p$ be any prime and use $\F_p$ coefficients for homology. 
		Consider the space $C(\C^*)\coloneqq\bigsqcup_{n\in\N}C_n(\C^*)$. Then we have:
		\[
		H_*(C(\C^*))=b\cdot H_*(C(\C))+[a,b]\cdot  H_*(C(\C))+[a,[a,b]]\cdot  H_*(C(\C))+\dots
		\]
	\end{lem}
	\begin{proof}
		We use Theorem \ref{thm: omologia spazi di configurazione etichettati} to compute $H_*(C(\C;S^0\vee S^0);\F_p)$ and then we identify $H_*(C(\C^*);\F_p)$ as the subspace spanned by classes involving exactly one white particle. Let us denote by $a$ (resp. $b$) the class in $H_0(S^0\vee S^0)$ which represent a black particle (resp. a white particle). The basic brackets involving only one white particle turns out to be $b$, $[a,b]$, $[a,[a,b]]$, $[a,[a,[a,b]]]$ etc. Now if $x$ is one of these brackets, $Q(x)$ will be a class containing $p$ white particles. So the classes of $H_*(C(\C^*);\F_p)$ are of the form $x\cdot y$, where $y\in H_*(C(\C);\F_p)$ and $x$ is one of the basic bracket listed above, and this proves the statement. 
	\end{proof}
	\begin{corollario}\label{cor:omologia nelle configurazioni nel piano bucato} Let $q$ be any natural number and take $\F_p$-coefficients for homology, for $p$ a fixed prime. Then
		\[
		H_*(C_q(\C^*))=b\cdot H_*(C_q(\C))+[a,b]\cdot  H_{*-1}(C_{q-1}(\C))+\dots+ [a,[a,\dots[a,b]]]\cdot H_{0}(C_{0}(\C))
		\]
	\end{corollario}
	\begin{defn}
		Let us denote by $d(q)$ the dimension of $H_*(C_q(\C);\F_p)$ as $\F_p$-vector space.
	\end{defn}
	Lemma \ref{lem: Z/p punti fissi} and Corollary \ref{cor:omologia nelle configurazioni nel piano bucato} allows us to rewrite equation (\ref{eq: configuration spaces}) in the following way:
	\begin{equation}\label{eq: equazione da dimostrare}
		d(pq)=d(q)+d(q-1)+\dots+d(1)+d(0)
	\end{equation}
	Before going into the details of the proof of this equation, let us look at an example. The general proof will be a generalization of the methods we are going to use in this specific case.
	\begin{es}
		Let $p=q=3$. To prove equation \ref{eq: equazione da dimostrare} one can proceed by induction on $q$, so by inductive hypothesis it suffices to show that $d(pq)=d(q)+d(p(q-1))$. Let us verify this equality in this specific case. $H_*(C_3(\C);\F_3)$ is generated by $\iota^3$ and $\iota u$, so $d(3)=2$. The generators of $H_*(C_9(\C);\F_3)$ are listed in the left table, while those of $H_*(C_6(\C);\F_3)$ are in the right one:
		\[
		\begin{tabular}{cc}
			\toprule
			Homology class  & Degree\\
			\midrule
			$\iota^9$ & $0$ \\
			$\iota^7u$ & $1$ \\
			$\iota^3\beta_1$ & $4$ \\
			$\iota u\beta_1$ & $5$ \\
			$\iota^3\alpha_1$ & $5$ \\
			$\iota u\alpha_1$ & $6$\\ 
			\bottomrule
		\end{tabular}
		\qquad
		\begin{tabular}{cc}
			\toprule
			Homology class  & Degree\\
			\midrule
			$\iota^6$ & $0$ \\
			$\iota^4u$ & $1$ \\
			$\beta_1$ & $4$ \\
			$\alpha_1$ & $5$ \\
			\bottomrule
		\end{tabular}
		\]
		Therefore $d(9)=6=d(6)+d(3)$ and the equality holds. A more conceptual proof of this equality is the following: observe that four classes of $H_*(C_9(\C);\F_3)$ are obtained just multiplying by $\iota^3$ the generators of $H_*(C_6(\C);\F_3)$. The remaining generators are $\iota u\beta_1$ and $\iota u\alpha_1$ and they can be obtained from those of $H_*(C_3(\C);\F_3)$ by the following change of variables:
		\begin{align*}
			u &\mapsto \alpha_1\\
			\iota^{2l+1}&\mapsto \iota^{p-2}u\beta_1^l
		\end{align*}
		This procedure can be useed to prove Equality \ref{eq: equazione da dimostrare} in general, as we will see in Proposition \ref{prop:identità tra le dimensioni}.
	\end{es}
	\begin{prop}\label{prop:identità tra le dimensioni}
		For any $p$ prime and $q\in\N$ we have $d(pq)=d(q)+d(q-1)+\dots+d(1)+d(0)$.
	\end{prop}
	\begin{proof}
		We proceed by induction on $q$. If $q=0$ there is nothing to prove. Let us suppose that the equation holds until $q$, let us prove it for $q+1$: by induction we have
		\begin{align}
			\sum_{i=0}^{q+1}d(i)=d(pq)+d(q+1)
		\end{align}
		Therefore it suffices to show that $d(pq)+d(q+1)=d(p(q+1))$. We will do this by constructing an explicit isomorphism of vector spaces between $H_*(C_{pq}(\C))\oplus H_*(C_{q+1}(\C))$ and  $H_*(C_{p(q+1)}(\C))$. Let us suppose $p$ is an odd prime (the case $p=2$ is similar). We refer to Remark \ref{oss:notazione per omologia} for the notation we are going to use. Consider the linear map
		\[
		f:H_*(C_{pq}(\C))\oplus H_*(C_{q+1}(\C))\to H_*(C_{p(q+1)}(\C))
		\]
		defined on the basis monomials as follows:
		\begin{itemize}
			\item If $x$ is a monomial of $H_*(C_{pq}(\C))$, then $f(x)=\iota^px$.
			\item If $x=\iota^ku^{\epsilon}\alpha_{i_1}\cdots\alpha_{i_m}\beta_{j_1}^{a_1}\cdots\beta_{j_n}^{a_n}$ is a monomial of $H_*(C_{q+1}(\C))$, with $\epsilon=0,1$, then 
			\[
			f(x)\coloneqq 
			\begin{cases}
				& \beta_1^l\alpha_1^{\epsilon}\alpha_{i_1+1}\cdots\alpha_{i_m+1}\beta_{j_1+1}^{a_1}\cdots\beta_{j_n+1}^{a_n} \quad\text{ if } k=2l\\
				&(\beta_1^lu\iota^{p-2})\alpha_1^{\epsilon}\alpha_{i_1+1}\cdots\alpha_{i_m+1}\beta_{j_1+1}^{a_1}\cdots\beta_{j_n+1}^{a_n} \quad\text{ if } k=2l+1
			\end{cases}
			\]
			In other words, $f(x)$ is the monomial of $H_*(C_{p(q+1)}(\C))$ obtained from $x$ by the following substitution of variables:
			
			\begin{itemize}
				\item[(a)] $\alpha_i\mapsto \alpha_{i+1}$
				\item[(b)] $\beta_i\mapsto \beta_{i+1}$
				\item[(c)] $u\mapsto \alpha_{1}$
				\item[(d)] $\iota^k\mapsto\begin{cases}
					\beta_1^l \text{ if } k=2l\\
					\beta_1^lu\iota^{p-2} \text{ if } k=2l+1
					
				\end{cases}$

			\end{itemize}

		\end{itemize}
		We claim that $f$ is an isomorphim of vector spaces:
		\begin{itemize}
			\item $f$ is well defined: clearly if we multiply by $\iota^p$ a monomial of $H_*(C_{pq}(\C))$ we obtain a monomial of $H_*(C_{p(q+1)}(\C))$. So let us pick  $x=\iota^ku^{\epsilon}\alpha_{i_1}\cdots\alpha_{i_m}\beta_{j_1}^{a_1}\cdots\beta_{j_n}^{a_n}\in H_*(C_{q+1}(\C))$; we prove that $f(x)\in H_*(C_{p(q+1)}(\C))$: by hypothesis we know that 
			\[
			q+1=k+2\epsilon+\sum_{r=1}^m 2p^{i_r}+\sum_{s=1}^na_s2p^{j_s}
			\]
			If $k=2l$, then $f(x)$ is a class which involves the following number of points:
			\[
			2lp+2p\epsilon+\sum_{r=1}^m 2p^{i_r+1}+\sum_{s=1}^na_s2p^{j_s+1}=p(q+1)
			\]
			proving our claim. If $k=2l+1$ the computation is analogous.
			\item $f$ is injective: $f$ restricted to the subspaces $H_*(C_{pq}(\C))$ and $H_*(C_{q+1}(\C))$ is injective by definition. Moreover the intersection of $f(H_*(C_{pq}(\C)))$ and $f(H_*(C_{q+1}(\C)))$ contains only the zero element: indeed the elements of  $f(H_*(C_{pq}(\C)))$ are sum of monomials which contain $\iota^p$, while the monomials spanning  $f(H_*(C_{pq}(\C)))$ do not contain $\iota^p$. 
			\item $f$ is surjective: to achieve this we need to prove that if $x$ is a basic monomial in $H_*(C_{p(q+1)}(\C))$, then it is of the following forms: 
			\begin{enumerate}
				\item $x=\iota^py$ for some $y\in H_*(C_{pq}(\C))$
				\item $x$ contains only the letters from $\{\alpha_i,\beta_i\}_{i\geq 1}$. In other words, $x$ do not contain $\iota $ and $u$. 
				\item $x=\iota^{p-2}uy$ with $y$ a monomial in $\{\alpha_i,\beta_i\}_{i\geq 1}$.
			\end{enumerate}
			But this is exactly the content of the following Proposition \ref{prop:classification of monomials}, therefore $f$ is surjective.
		\end{itemize}
		
	\end{proof}
	
	\begin{prop}\label{prop:classification of monomials}
		Let $p$ be an odd prime, $n\in\N$ and $n= k$ mod $p$. If $x\in H_*(C_{n}(\C);\F_p)$ is a basic monomial, then it has one of the following forms:
		\begin{enumerate}
			\item $x=\iota^py$ for some $y\in H_*(C_{n-p}(\C))$
			\item$x=\iota^k\alpha_{i_1}\cdots\alpha_{i_m}\beta_{j_1}^{a_1}\cdots\beta_{j_s}^{a_s}$
			\item $x=(\iota^{k-2}u)\alpha_{i_1}\cdots\alpha_{i_m}\beta_{j_1}^{a_1}\cdots\beta_{j_s}^{a_s}$
			
		\end{enumerate}
		If $k=0$ (resp. $k=1$) replace the exponent $k-2$ in point $(c)$ with the corresponding class in $\Z/p$, i.e with $p-2$ (resp. $p-1$) to get the correct statement.   
	\end{prop}
	\begin{proof}
		We proceed by induction on $n$: 
		\begin{itemize}
			\item If $n=1$ the only class in $H_*(C_{1}(\C);\F_p)$ is $\iota$, so the statement is true. If $n\in\{2,\dots,p\}$ then $H_*(C_{n}(\C);\F_p)$ contains only two classes, $\iota^{n}$ and $\iota^{n-2}u$, and the statement follows.
			\item Assume the result is true until $n$, let us prove it for $n+p$: let
			\[
			x=\iota^lu^{\epsilon}\alpha_{i_1}\cdots\alpha_{i_m}\beta_{j_1}^{a_1}\cdots\beta_{j_s}^{a_s}\in H_*(C_{n+p}(\C))
			\]
			If $l\geq p $ then we are in the first case. So let us restrict to the case where $l\leq p-1$. If $x$ contains some $\beta_i$ then $x=\beta_i\cdot x'$ with $x'\in H_*(C_{n+p-2p^i}(\C))$. Since $i\geq 1$ we have $n+p-2p^i\leq n$. Moreover $n+p-2p^i= k$ mod $p$ therefore we can use the inductive hypothesis for $x'$ and get the result. We can proceed in the same way when $x$ contains one of the variables $\{\alpha_i\}_{i\geq 1}$. The only case which is not yet considered is when $x=\iota^lu^{\epsilon}$. In this case $n+p$ must be equal to $l+2\epsilon$. Therefore $l=n+p-2\epsilon\leq p-1$ if and only if $n\leq 2\epsilon-1\leq 1$ which is not the case we are considering.
		\end{itemize}
	\end{proof}
	This concludes the proof of Theorem \ref{thm:degeneration of a spectral sequence} in the case $p|n$. It remains to prove the statement for $n=pq+1$. This is equivalent to show that
	\[
	d(pq+1)=d(q)+d(q-1)+\dots+d(0)=d(pq)
	\]
	Where the last equality holds for Proposition \ref{prop:identità tra le dimensioni}. But this is an easy consequence of Proposition \ref{prop:classification of monomials}:
	\begin{corollario}
		If $p$ divides $n$, then $H_*(C_n(\C);\F_p)$ and $H_*(C_{n+1}(\C);\F_p)$ have the same dimension as $\F_p$-vector spaces.
	\end{corollario}
	\begin{proof}
		We observe that multiplication by $\iota$ is an isomorphism between $H_*(C_n(\C);\F_p)$ and $H_*(C_{n+1}(\C);\F_p)$. Clearly it is injective. By Proposition \ref{prop:classification of monomials} we have that each monomial in $H_*(C_{n+1}(\C);\F_p)$ contains at least one $\iota$, and this proves the surjectivity. 
	\end{proof}
	
	\subsection{Computation of $H_*^{S^1}(C_n(\C);\F_p)$ when $n=0,1$ mod $p$}\label{sec:omologia equivariante quando p divie n o n-1}
	In this section we compute $H_*^{S^1}(C_n(\C);\F_p)$ when $n=0,1$ mod $p$.
	\begin{thm}\label{thm: calcolo dell'omologia equivariante degli spazi di configurazione quando p divide n}
		Let $p$ be a prime, $n\in\N$ such that $n=0,1$ mod $p$. Then 
		\[
		H_*^{S^1}(C_n(\C);\F_p)\cong H_*(C_n(\C);\F_p)\otimes H_*(BS^1;\F_p)
		\]
	\end{thm}
	\begin{proof}
		Consider the map of fibrations
		\[
		\begin{tikzcd}
			& C_n(\C)\arrow[d,hook]\arrow[r] & C_{n}(\C)\arrow[d,hook]\\
			& C_n(\C)_{\Z/p} \arrow[r] \arrow[d]&  C_{n}(\C)_{S^1}  \arrow[d]\\
			& B\Z/p\arrow[r] & BS^1
		\end{tikzcd}
		\]
		Then observe that with coefficients in $\F_p$ this map induces a surjection between the $E^2$ pages of the homological spectral sequences. The result now follows from Theorem \ref{thm:degeneration of a spectral sequence}.
	\end{proof}
	\begin{oss}
		It would be interesting to compute the ring structure of $H^*_{S^1}(C_n(\C);\F_p)$. However there are some non-trivial extension problems to solve. For example, if we put $n=p=2$ the homotopy quotient $C_2(\C)_{S^1}$ is a model for $B(\Z/2)$, therefore 
		\[
		H^*_{S^1}(C_2(\C);\F_2)=\F_2[x]
		\]
		where $x$ is a variable of degree one. Theorem \ref{thm: calcolo dell'omologia equivariante degli spazi di configurazione quando p divide n} tells us that 
		\[
		H^*_{S^1}(C_2(\C);\F_2)\cong \frac{\F_2[x]}{(x^2)}\otimes \F_2[c]
		\]
		as $\F_2[c]$-module, where $x$ (resp. $c$) is a generator of $H^1(C_2(\C);\F_2)$ (resp. of $H^2(BS^1;\F_2)$). To get the correct ring structure we need to impose the relation $x^2=c$. 
	\end{oss}
	\subsection{Computation of $H^{S^1}_*(C_n(\C);\F_p)$ when $n\neq 0,1$ mod $p$}\label{sec: calcoli quando p non divide n}
	In this section we compute $H_*^{S^1}(C_n(\C);\F_p)$ when $n\neq 0,1$ mod $p$. Of course this request is not empty only if $p$ is an odd prime. The main result is the following:
	\begin{thm}\label{thm:calcolo omologia equivariante quando p non divide n o n-1}
		Let $p$ be an odd prime, $n\in\N$ such that $n\neq 0,1$ mod $p$. Then
		\[
		H_*^{S^1}(C_n(\C);\F_p)\cong coker(\Delta)
		\]
		where $\Delta:H_*(C_n(\C);\F_p)\to H_{*+1}(C_n(\C);\F_p)$ is the $BV$-operator.
	\end{thm}
	\begin{proof}
		Consider the homological Serre spectral sequence associated to the fibration 
		\[
		C_n(\C)\to C_n(\C)_{S^1}\to BS^1
		\]
		By Proposition \ref{prop: differenziale dato da delta} the second page is given by
		\[
		E^2_{i,j}=H_i(C_n(\C))\otimes H_j(BS^1) \qquad d^2(x\otimes y_{2j})=\begin{cases}
			\Delta(x)\otimes y_{2j-2} \text{ if } j\geq 1\\
			0 \text{ if } j=0
		\end{cases}
		\]
		where $y_{2j}$ is the generator of $H_{2j}(BS^1)$. By Theorem \ref{thm: omologia spazi di configurazione etichettati} any class of $H_*(C_n(\C);\F_p)$ is a product of the variables $\iota$, $[\iota,\iota]$, $Q^i[\iota,\iota]$ and $\beta Q^i[\iota,\iota]$, $i\geq 1$. In particular any class is of the form $\iota^k[\iota,\iota]^{l}x$, where $k\in\N$, $l=0,1$  and $x$ is a monomial which contains only the letters $\{Q^i[\iota,\iota],\beta Q^i[\iota,\iota]\}_{i\geq 1}$. We claim that the operator $\Delta$ acts as follows:
		\begin{equation}\label{eq: formule di delta sulla base dei monomi}
			\Delta(\iota^kx)=k(k-1)\iota^{k-2}[\iota,\iota]x \qquad \Delta(\iota^k[\iota,\iota]x)=0
		\end{equation}
		A detailed proof of these formulas will be given in Section \ref{sec: proprietà della delta} (Proposition \ref{prop: calcolo di delta su specifici monomi}). Observe that if we have a monomial of the form $\iota^k x$, $x\in H_*(C_m(\C);\F_p)$, then $n=k+m=k$ mod $p$ ($x$ contains only letters from $\{Q^i[\iota,\iota],\beta Q^i[\iota,\iota]\}_{i\geq 1}$, so the number of points $m$ is divisible by $p$). Therefore $k\neq 0,1$ mod $p$ and the term $k(k-1)\iota^{k-2}[\iota,\iota]x$ is never zero. Now we can conclude: the formula \ref{eq: formule di delta sulla base dei monomi} shows that the third page of the Serre spectral sequence looks as follows: $E^3_{i,j}=0$ for any $j\geq 1$, while the first column $E^3_{0,*}$ may contain some non zero elements. To be more precise, $E^3_{0,*}$ is the quotient of $E^2_{0,*}=H_*(C_n(\C);\F_p)$ by the image of the differential. Since the differential of the second page is given by $\Delta$, we get that 
		\[
		E^3_{0,*}\cong coker(\Delta)
		\]
		The third page contains only the first column, so the spectral sequence degenerates and we get the statement.    
	\end{proof}
	\begin{oss}
		We can explicitly describe $coker(\Delta)$: a basis of this vector space is given by (the image of) classes in $H_*(C_n(\C);\F_p)$ which do not contain the bracket $[\iota,\iota]$ (Equation \ref{eq: formule di delta sulla base dei monomi}).
	\end{oss}
	\subsection{Auxiliary computations}\label{sec: proprietà della delta}
	
	In this paragraph we recollect some easy algebraic computations which are relevant for Section \ref{sec: calcoli quando p non divide n}. In particular we prove the formulas \ref{eq: formule di delta sulla base dei monomi} in Proposition \ref{prop: calcolo di delta su specifici monomi}. In what follows $p$ will be an odd prime. We begin by recalling some basic properties of the bracket and $\Delta$; to ease the notation we will write $(-1)^x$ instead of $(-1)^{deg(x)}$, where $x$ is an element in some graded vector space. We refer to \cite{Cohen} and \cite{Getzler93} for further details. 
	\begin{enumerate}
		\item Graded anticommutativity: $[x,y]=(-1)^{xy+x+y}[y,x]$.
		\item Jacobi relation: $[x,[y,z]]=[[x,y],z]-(-1)^{y+x+xy}[y,[x,z]]$.
		\item The bracket is a derivation:     $[x,yz]=[x,y]z+(-1)^{y+yx}y[x,z]$.
		\item $\Delta(xy)=\Delta(x)y+(-1)^x x\Delta(y)+(-1)^x[x,y]$.
		\item $\Delta[x,y]=[\Delta x,y]+(-1)^{x+1}[x,\Delta y]$. 
		\item $[x,Qy]=ad^p(y)(x)$, where $ad(y)(x)\coloneqq[x,y]$ and for any $n\in\N$ we define $ad^n(y)(x)\coloneqq ad(y)(ad^{n-1}(y)(x))$. For example, $ad^2(y)(x)=[[x,y],y]$, $ad^3(y)(x)=[[[x,y],y],y]$ and so on.
		\item $[x,\beta Qy]=[x,ad^{p-1}(y)(\beta y)]$
	\end{enumerate}
	\begin{lem}\label{lem: calcolo di delta su iota alla k}
		Let $p$ be an odd prime, $k\in\N$. Then $\Delta(\iota^k)=k(k-1)\iota^{k-2}[\iota,\iota]$.
	\end{lem}
	\begin{proof}
		It suffices to proceed by induction: $\Delta(\iota)=0$ since $\iota$ is the top class of $H_*(C_1(\C));\F_p)$. $\Delta(\iota^2)=[\iota,\iota]$ by equation $(4)$ at the beginning of this section. The general formula follows easily using equation $(4)$.
	\end{proof}
	\begin{lem}\label{lem:delta fa zero sui monomi divisibili per p}
		Let  $x\in H_*(C(\C);\F_p)$ be a monomial containing only the letters $\{Q^i[\iota,\iota],\beta Q^i[\iota,\iota]\}_{i\geq 1}$. Then $\Delta(x)=0$ and $\Delta(\iota x)=0$.
	\end{lem}
	\begin{proof}
		We prove that $\Delta(x)=0$, the other case is analogous. Since $x$ contains only the letters $\{Q^i[\iota,\iota],\beta Q^i[\iota,\iota]\}_{i\geq 1}$, it is a class of $H_*(C_n(\C);\F_p)$ for some $n$ divisible by $p$. Theorem \ref{thm: calcolo dell'omologia equivariante degli spazi di configurazione quando p divide n} tell us that the homological Serre spectral sequence associated to 
		\[
		C_n(\C)\to C_n(\C)_{S^1}\to BS^1
		\]
		degenerates at the second page. But the differential of the second page is given by $\Delta$ (Proposition \ref{prop: differenziale dato da delta}) so $\Delta(x)=0$.
	\end{proof}
	\begin{lem}\label{lem: iota alla k bracket x fa zero}
		Let  $x\in H_*(C(\C);\F_p)$ be a monomial containing only the letters $\{Q^i[\iota,\iota],\beta Q^i[\iota,\iota]\}_{i\geq 1}$. Then $[\iota^k,x]=0$.
	\end{lem}
	\begin{proof}
		Since the bracket is a derivation it is enough to prove that $[\iota,x]=0$. By Lemma \ref{lem:delta fa zero sui monomi divisibili per p} $\Delta(\iota x)=0$, therefore
		\[
		0=\Delta(\iota x)=\Delta(\iota)x\pm\iota\Delta(x)\pm[\iota,x]= \pm[\iota,x]
		\]
		The last equality holds because $\Delta(\iota)=0$ ($\iota$ is the top class of $H_*(C_1(\C);\F_p)$) and $\Delta(x)=0$ (Lemma \ref{lem:delta fa zero sui monomi divisibili per p}).
	\end{proof}
	
	\begin{lem}\label{lem:bracket tra iotaiota e x fa zero}
		Let $x\in H_*(C(\C),\F_p)$ be a monomial which contains only the letters $\{Q^i[\iota,\iota],\beta Q^i[\iota,\iota]\}_{i\geq1}$. Then $[[\iota,\iota],x]=0$.
	\end{lem}
	\begin{proof}
		Since the bracket is a derivation it is enough to show that $[[\iota,\iota],Q^i[\iota,\iota]]=0$ and that $[[\iota,\iota],\beta Q^i[\iota,\iota]]=0$. In the first case we proceed by induction: if $i=1$ we get $[[\iota,\iota],Q[\iota,\iota]]=ad^p([\iota,\iota])([\iota,\iota])$ by property (6), and the latter term is zero by Jacobi. In general we use again property (6) and we get:
		\begin{align*}
			[[\iota,\iota],Q^i[\iota,\iota]]=ad^p(Q^{i-1}[\iota,\iota])([\iota,\iota])=0
		\end{align*} 
		where the last equality holds by induction. The other formula can be proved as follows: by equation (7) we get
		\[
		[[\iota,\iota],\beta Q^{i}[\iota,\iota]]=[[\iota,\iota],ad^{p-1}(Q^{i-1}[\iota,\iota])(\beta Q^{i-1}[\iota,\iota])]
		\]
		One of the brackets contained in $ad^{p-1}(Q^{i-1}[\iota,\iota])(\beta Q^{i-1}[\iota,\iota])$ is $[\beta Q^{i-1}[\iota,\iota],Q^{i-1}[\iota,\iota]]$ which we claim is zero: 
		\begin{align*}
			0&=\Delta(\beta Q^{i-1}[\iota,\iota]Q^{i-1}[\iota,\iota])\\
			&=\Delta(\beta Q^{i-1}[\iota,\iota])Q^{i-1}[\iota,\iota]\pm \beta Q^{i-1}[\iota,\iota]\Delta(Q^{i-1}[\iota,\iota])\pm [\beta Q^{i-1}[\iota,\iota],Q^{i-1}[\iota,\iota]]\\
			&= \pm[\beta Q^{i-1}[\iota,\iota],Q^{i-1}[\iota,\iota]]
		\end{align*}
		where in the first and third equality we used Lemma \ref{lem:delta fa zero sui monomi divisibili per p}.
	\end{proof}
	%\begin{lem}\label{lem:calcoli tecnici del bracket}
	%    Let $x\in H_*(C(\C),\F_p)$ be a monomial which contains only the letters $\{Q^iu,\beta Q^iu\}_{i\geq1}$. Then $[\iota^k,x]=0$ and $[\iota^{k-2}u,x]=0$. 
	%\end{lem}
	%\begin{proof}
	%   Induction on $k$ and property (3) give the statement.
	%\end{proof}
	%We finish this section by computing $\Delta$ on some generators of $H_*(C(\C);\F_p)=\F_p[\iota,\beta Qu,\beta Q^2u,\dots,]\otimes \Lambda[u,Qu,Q^2u,\dots]$.
	\begin{prop}\label{prop: calcolo di delta su specifici monomi}
		Let $n\in\N$  and suppose $n\neq 0,1$ mod $p$. Let $\iota^kx$ and $\iota^k[\iota,\iota]x$ be classes is $H_*(C_n(\C);\F_p)$, where $x$ be a monomial which contains only the variables $\{Q^i[\iota,\iota], \beta Q^i[\iota,\iota]\}_{i\geq 1}$. Then
		\[
		\begin{cases}
			\Delta(\iota^kx)=k(k-1)\iota^{k-2}[\iota,\iota]x\\
			\Delta(\iota^{k}[\iota,\iota]x)=0  
		\end{cases}
		\]
	\end{prop}
	\begin{proof}
		By Lemma \ref{lem: calcolo di delta su iota alla k} we have $\Delta(\iota^k)=k(k-1)\iota^{k-2}[\iota,\iota]$. Using property (4) of $\Delta$ together with Lemma \ref{lem:delta fa zero sui monomi divisibili per p} and Lemma \ref{lem: iota alla k bracket x fa zero} we get
		\[
		\Delta(\iota^k x)=\Delta(\iota^k)x\pm \iota^k\Delta(x)\pm [\iota^k,x]=k(k-1)\iota^{k-2}[\iota,\iota]
		\]
		and this proves the first part of the statement. Similarly, 
		\[
		\Delta(\iota^{k}[\iota,\iota] x)=\Delta(\iota^{k}[\iota,\iota])x\pm \iota^{k}[\iota,\iota]\Delta(x)\pm [\iota^{k}[\iota,\iota],x]
		\]
		The last term is zero since the bracket is a derivation and $[\iota^k,x]=0=[[\iota,\iota],x]$ (Lemma \ref{lem: iota alla k bracket x fa zero} and Lemma \ref{lem:bracket tra iotaiota e x fa zero}).  The middle term is zero by Lemma \ref{lem:delta fa zero sui monomi divisibili per p}. The first term is zero as well for the following reason:
		\[
		\Delta(\iota^{k}[\iota,\iota])=k(k-1)\iota^{k-2}[\iota,\iota][\iota,\iota]\pm \iota^k\Delta[\iota,\iota]\pm[\iota^k,[\iota,\iota]]
		\]
		The first term is zero since $[\iota,\iota]$ is a variable of odd degree, so it squares to zero. The second term is zero since $[\iota,\iota]$ is the top class of $H_*(C_2(\C);\F_p)$. To see that the last term is zero just use the fact that the bracket is a derivation and that $[\iota,[\iota,\iota]]=0$ (if $p\neq 3$ the Jacobi relation imply $[\iota,[\iota,\iota]]=0$, while for $p=3$ this iterated bracket is zero by definition).
	\end{proof}
	
	\section{Operations for odd degree classes}\label{sec: operations for odd degree classes}
	
	As we saw in Section \ref{sec: equiavariant operations} the homology operations for odd degree elements in a gravity algebra are governed by $H_*^{\Sigma_n}(\M_{0,n+1};\F_p(\pm 1))\cong H_*(B_n/Z(B_n);\F_p(\pm 1))$, where $p$ is any fixed prime. In this Section we compute this homology (Theorem \ref{thm:collasso sequenza spettrale per i labelled fiberwise config. spaces}). Observe that if $p=2$ the sign representation is actually the trivial representation, therefore we get an alternative computation of $H_*^{\Sigma_n}(\M_{0,n+1};\F_2)$. The techniques involved comes from the theory of fiberwise configurations spaces, which we will quickly review in Section \ref{sec: fiberwise configuration space} (further details can be found in \cite{Cohen-Maldonado} and \cite{Cohen-Bodigheimer}). 
	\subsection{Fiberwise configuration spaces}\label{sec: fiberwise configuration space}
	Let $\lambda:E\to B$ be a fiber bundle with fiber $Y$. We consider the following space of \textbf{(ordered) fiberwise configurations} of points
	\[
	E(\lambda,n)\coloneqq\{(e_1,\dots,e_n)\in E^n\mid e_i\neq e_j \text{ and } \lambda(e_i)=\lambda(e_j) \text{ if } i\neq j\}
	\]
	The symmetric group $\Sigma_n$ acts on $E(\lambda,n)$ by permuting the coordinates, so we can also define the space of \textbf{unordered fiberwise configurations} as the quotient $E(\lambda,n)/\Sigma_n$. In particular there are fiber bundles
	\begin{align*}
		& F_n(Y)\hookrightarrow E(\lambda,n)\to  B
		& C_n(Y)\hookrightarrow E(\lambda,n)/\Sigma_n\to B
	\end{align*}
	Now let $X$ be a connected CW-complex with basepoint $\ast$. We consider the following space of \textbf{fiberwise configurations with label in} $X$
	\begin{equation*}
		E(\lambda;X)\coloneqq \bigsqcup_{n=0}^{\infty}E(\lambda,n)\times_{\Sigma_n}X^n/\sim 
	\end{equation*}
	where $\sim$ is the equivalence relation determined by 
	\[
	(e_1,\dots,e_n)\times(x_1,\dots,x_n)\sim (e_1,\dots,\hat{e}_i,\dots,e_n)\times(x_1,\dots,\hat{x}_i,\dots,x_n)
	\]
	when $x_i=\ast$.
	\begin{oss}
		If the fiber bundle is given by a constant map $\lambda:E\to \{\ast\}$ the space $E(\lambda;X)$ is usually denoted by $C(E;X)$. These are just configuration of points on $E$ with label in $X$. 
	\end{oss}
	The spaces $E(\lambda;X)$ are equipped with a natural filtration by the number of points
	\[
	\{\ast\}\subseteq E_1(\lambda;X)\subseteq E_2(\lambda;X)\subseteq\dots\subseteq E(\lambda;X)
	\]
	where $E_k(\lambda;X)$ is the subspace 
	\[
	E_k(\lambda;X)\coloneqq \bigsqcup_{n=0}^{k}E(\lambda,n)\times_{\Sigma_n}X^n/\sim 
	\]
	In the literature it is standard to denote the quotient $E_k(\lambda;X)/E_{k-1}(\lambda;X)$ by $D_k(\lambda;X)$.
	
	\begin{thm}[[Proposition $2.3$ of \cite{Cohen-Bodigheimer}]\label{thm:stable splitting}
		Let $X$ be a connected CW-complex and $F\hookrightarrow E\xrightarrow{\lambda} B$ be a fiber bundle. Then $E(\lambda;X)$ is stably equivalent to $\bigvee_{k\in\N}D_k(\lambda;X)$. In particular we have a homology isomorphism
		\[
		\Tilde{H}_*(E(\lambda;X);\Z)\cong\bigoplus_{k=1}^{\infty}\Tilde{H}_*(D_k(\lambda;X);\Z)
		\]
		
	\end{thm}
	\begin{prop}[\cite{Cohen-Maldonado}, p. 8]\label{prop:quasi-iso al livello delle catene}
		Let $X$ be a connected CW complex with basepoint $\ast\in X$ and $\mathbb{F}$ be any field. Then we have a quasi-isomorphism
		\[
		C_*(E(\lambda,k))\otimes_{\Sigma_k}\Tilde{H}_*(X)^{\otimes k}\to C_*(D_k(\lambda;X))
		\]
		where $C_*(-)$ are the singular chains with $\mathbb{F}$-coefficients and $\Sigma_k$ acts on the graded $\mathbb{F}$-vector space $\Tilde{H}_*(X)^{\otimes k}$ by permutation of variables with the usual sign convention.
	\end{prop}
	%\begin{proof}
	%   Since we are working with coefficients on a field $\mathbb{F}$ one can always find a quasi-isomorphism $j:H_*(X)\to C_*(X)$. This induces a quasi-isomorphism
	%  \[
	%  f:C_*(E(\lambda,k))\otimes_{\Sigma_k}H_*(X)^{\otimes k}\to C_*(E(\lambda,k)\times_{\Sigma_k}X^k)
	%\]
	%Since $X$ is a path-connected space with non-degenerate basepoint we have a cofibration
	%\[
	%E(\lambda,k)\times_{\Sigma_k}X_{k}\hookrightarrow E(\lambda,k)\times_{\Sigma_k} X^k\to D_k(\lambda;X)
	% \]
	%where $X_k\coloneqq\{(x_1,\dots,x_k)\in X^k\mid x_i=\ast \text{ for some } i\}$. Therefore $f$ induces a quasi isomorphism 
	%\[
	% C_*(E(\lambda,k))\otimes_{\Sigma_k}\Tilde{H}_*(X)^{\otimes k}\to C_*\left(\frac{E(\lambda,k)\times_{\Sigma_k}X^k}{E(\lambda,k)\times_{\Sigma_k}X_k}\right)=C_*(D_k(\lambda;X))
	%\]
	%and we get the statement.
	%\end{proof}
	
	\subsection{A model for $B(B_n/Z(B_n))$ with fiberwise configuration spaces}\label{sec: model of B_n/Z as fiberwise configuration space}
	Let $\C\hookrightarrow E\xrightarrow{\lambda} \C P^{\infty}$ be the tautological line bundle. Explicitly, the total space is 
	\[
	E\coloneqq\{(v,l)\in \C^{\infty}\times \C P^{\infty}\mid v\in l\}
	\]
	and $\lambda: E\to \C P^{\infty}$ is the projection on the second coordinate. The most important observation of this section is the following proposition:
	\begin{prop}\label{prop:quozienti omotopici come fiberwise configurations}
		Let $n\geq 2$. The unordered fiberwise configurations $E(\lambda,n)/\Sigma_n$ is a model for the classifying space  $B(B_n/Z(B_n))$. Similarly, $E(\lambda,n)$ is a model for the classifying space $B(PB_n/Z(PB_n))$. 
	\end{prop}
	\begin{proof}
		We prove the statement for $E(\lambda,n)/\Sigma_n$, the other case is analogous. Consider the fibration $C_n(\C)\hookrightarrow E(\lambda,n)/\Sigma_n\to \C P^{\infty}$. The long exact sequence for homotopy groups shows that $\pi_i(E(\lambda,n)/\Sigma_n)=0$ for all $i\geq 3$. Moreover, we get the following exact sequence:
		\[
		\begin{tikzcd}[column sep=small]
			&0\arrow[r] &\pi_2(E(\lambda,n)/\Sigma_n) \arrow[r] & \Z \arrow[r, "\partial"] & B_n\arrow[r] &\pi_1(E(\lambda,n)/\Sigma_n)\arrow[r] & 0
		\end{tikzcd}
		\]
		Now we claim that the connecting homomorphism $\partial$ includes $\Z$ as the center of $B_n$. To prove this, let us consider the map 
		\begin{align*}
			f:S^{\infty} &\to E(\lambda,n)/\Sigma_n\\
			v&\mapsto \{(v,l_v), (\zeta\cdot v,l_v),\dots,(\zeta^{n-1}\cdot v,l_v)\}
		\end{align*}
		where $\zeta\coloneqq e^{2\pi i/n}$ acts on $S^{\infty}\subseteq \C^{\infty}$ by multiplication and $l_v$ denotes the line spanned by $v$. This is a map of fibrations
		\[
		\begin{tikzcd}
			& C_n(\C) \arrow[d,hook] & S^1 \arrow[l] \arrow[d,hook]\\
			& E(\lambda,n)/\Sigma_n \arrow[d] & S^{\infty} \arrow[l,"f"] \arrow[d]\\
			& \C P^{\infty} & \C P^{\infty} \arrow[l]
		\end{tikzcd}
		\]
		so we get the following commutative diagram, whose rows are exact:
		\[
		\begin{tikzcd}[column sep=small]
			&0\arrow[r] &\pi_2(E(\lambda,n)/\Sigma_n) \arrow[r] & \Z \arrow[r, "\partial"] & B_n\arrow[r] &\pi_1(E(\lambda,n)/\Sigma_n)\arrow[r] & 0\\
			&0\arrow[r] \arrow[u] &0 \arrow[r] \arrow[u] & \Z \arrow[u,"id"]\arrow[r, "\cong"] & \Z \arrow[r]\arrow[u,"f_*"]&0\arrow[r] \arrow[u]& 0 \arrow[u]
		\end{tikzcd}
		\]
		Finally observe that $f_*$ includes $\Z$ as the center of $B_n$, so the same holds for $\partial$.
	\end{proof}
	\begin{oss}
	 Proposition \ref{prop:quozienti omotopici come fiberwise configurations} is not true for $n=1$. Indeed in this case $B_1/Z(B_1)=PB_1/Z(PB_1)$ is the trivial group and $E(\lambda,1)/\Sigma_1=E(\lambda,1)$ is the total space $E$ of the tautological line bundle $\lambda:E\to \C P^{\infty}$, which is not contractible.
	\end{oss}
	\begin{oss}
		The Braid group $B_n$ is equipped with a natural morphism to $\Sigma_n$, whose kernel is the pure braid group $PB_n$. Since this morphism sends the generator of the center $\delta^2$ to the identity permutation, there is a factorization
		\[
		\begin{tikzcd}
			&1\arrow[r] & PB_n \arrow[d] \arrow[r,hook]& B_n \arrow[d]\arrow[r] & \Sigma_n\arrow[r]\arrow[d] & 1\\
			&1\arrow[r] & PB_n/Z(PB_n)\arrow[r,hook]& B_n/Z(B_n) \arrow[r] & \Sigma_n\arrow[r] & 1\\
		\end{tikzcd}
		\]
		Therefore we can regard any $\Sigma_n$-module as a $B_n/Z(B_n)$ module.
	\end{oss}
	
	Now let $V$ be a graded vector space over some field $\F$, and assume that $V$ in concentrated in degree greater than $1$. In this case we can always find a bouquet of spheres $S_V$ such that
	\[
	V\cong \Tilde{H}_*(S_V;\F)
	\]
	We assume that the symmetric group $\Sigma_n$ acts on $V^{\otimes n}$ with the usual sign conventions. By the previous remark we can see $V^{\otimes n}$ as a $B_n/Z(B_n)$-module.

	\begin{prop}\label{prop: interpretazione come fiberwise configurations}
		Let $\F$ be any field, $q\in\N$, $n\geq 2$. Then we have an isomorphism
		\[
		H_*(B_n/Z(B_n);V^{\otimes n})\cong \Tilde{H}_{*}(D_n(\lambda;S_V);\F)
		\]
		In particular, if we choose $V$ to be a copy of $\F$ concentrated in degree $2q+1$ (resp. $2q$) we get
		\begin{align}
			& H_*(B_n/Z(B_n);\F(\pm 1))\cong H_{*+(2q+1)n}(D_n(\lambda;S^{2q+1});\F)\\
			& H_*(B_n/Z(B_n);\F)\cong H_{*+2qn}(D_n(\lambda;S^{2q});\F)
		\end{align} 
		
	\end{prop}
	\begin{proof}
		To ease the notation, let us call $G_n\coloneqq B_n/Z(B_n)$ and  $H_n\coloneqq PB_n/Z(PB_n)$. By definition $H_*(G_n; V^{\otimes n})$ is computed by $C_*^{cell}(EG_n)\otimes_{G_n}V^{\otimes n}$, where if $G$ is a discrete group $C_*^{cell}(EG_n)$ denotes the standard resolution of $\Z$ over $\Z[G]$ . Since $H_n$ is a subgroup of $G_n$, we can take $EG_n$ as a model for $EH_n$. Therefore:
		\[
		C_*^{cell}(EG_n)\otimes_{G_n}V^{\otimes n}\cong \frac{C_*^{cell}(EG_n)\otimes_{H_n}V^{\otimes n}}{\Sigma_n}\cong C_*^{cell}(EG_n)_{H_n}\otimes_{\Sigma_n}V^{\otimes n}
		\]
		where the last isomorphism holds because $H_n$ acts trivially on $V^{\otimes n}$. Note that $C_*^{cell}(EG_n)_{H_n}$ computes the homology of $H_n$. Proposition \ref{prop:quozienti omotopici come fiberwise configurations} tell us that $E(\lambda,n)$ is a model for the classifying space of $PB_n/Z(PB_n)$ and Proposition \ref{prop:quasi-iso al livello delle catene} give us the quasi-isomorphism
		\[
		C_*(E(\lambda,n))\otimes_{\Sigma_n}\Tilde{H}_*(S_V)^{\otimes n}\to C_*(D_n(\lambda;S_V))
		\]
		which allows us to conclude the proof.
	\end{proof}
	
	\subsection{Computations}
	By Proposition \ref{prop: interpretazione come fiberwise configurations} and Theorem \ref{thm:stable splitting} the computation of $H_*(B_n/Z(B_n);\F_p(\pm 1))$ is reduced to the computation of $H_*(E(\lambda;S^{2q+1});\F_p)$. 
	
	\begin{thm}\label{thm:collasso sequenza spettrale per i labelled fiberwise config. spaces}
		Let $p$ be an odd prime. Then
		\[
		H_*(E(\lambda;S^{2q+1});\F_p)\cong H_*(C(\C;S^{2q+1});\F_p)\otimes H_*(BS^1;\F_p)
		\]
	\end{thm}
	\begin{proof}
		We get the statement by proving that the Serre spectral sequence with $\F_p$-coefficients associated to the fibration $C(\C;S^{2q+1})\hookrightarrow E(\lambda;S^{2q+1})\to BS^1$ degenerates at page $E_2$. In particular, it suffices to show that the map induced by the inclusion $i_*:H_*(C(\C;S^{2q+1});\F_p)\to H_*(E(\lambda;S^{2q+1});\F_p)$ is injective. Consider the following map 
		\begin{align*}
			\psi:E(\lambda;S^{2q+1})&\to C(\C^{\infty};S^{2q+1})\\
			[((v_1,l),\dots,(v_n,l))\times (p_1,\dots,p_n)]&\mapsto [(v_1,\dots,v_n)\times (p_1,\dots,p_n)]
		\end{align*}
		where $(v_i,l)$ are points in the total space of the tautological line bundle $\C\hookrightarrow E\to \C P^{\infty}$, and $p_i\in S^{2q+1}$ are the corresponding labels. Now observe that the inclusion $j:C(\C;S^{2q+1})\to C(\C^{\infty};S^{2q+1})$ factors through $E(\lambda;S^{2q+1})$:
		\[
		\begin{tikzcd}
			&C(\C;S^{2q+1})\arrow[r,"j"] \arrow[d,"i"] &C(\C^{\infty};S^{2q+1}) \\
			& E(\lambda;S^{2q+1})\arrow[ur,"\psi"]
		\end{tikzcd}
		\]
		Since $j$ induces an injective map in mod $p$ homology (see \cite{Cohen} or \cite{Salvatore1}), the above commutative diagram shows that $i_*$ is injective as well.
	\end{proof}
	\begin{oss}
		If we replace $S^{2q+1}$ with $S^{2q}$ the proof written above does not work, indeed the map $i_*:H_*(C(\C;S^{2q});\F_p)\to H_*(C(\C^{\infty};S^{2q});\F_p)$ is not injective anymore, except when $p=2$. For details about the homology of labelled configuration spaces we refer to the work of F. Cohen \cite{Cohen}.
	\end{oss}

	\begin{thm}\label{thm: calcolo fiberqise configurations nel caso p=2}
		Let $q\geq 1$. Then
		\[
		H_*(E(\lambda;S^{2q});\F_2)\cong H_*(C(\C;S^{2q});\F_2)\otimes H_*(BS^1;\F_2)
		\]
	\end{thm}
	\begin{proof}
		Follow the proof of Theorem \ref{thm:collasso sequenza spettrale per i labelled fiberwise config. spaces}, just replace $2q+1$ with $2q$ and $p$ with $2$.
	\end{proof}	
	We end this section with a proof of Theorem \ref{thm: calcolo dell'omologia equivariante degli spazi di configurazione quando p divide n} based on labelled configuration spaces:
	\begin{corollario}\label{thm: collasso sequenza spettrale nel caso p=2}
		For any $n\in\N$ we have
		\[
		H_*^{S^1}(C_n(\C);\F_2)\cong H_*(C_n(\C);\F_2)\otimes H_*(BS^1;\F_2)
		\]
	\end{corollario}
	\begin{proof}
		By Proposition \ref{prop: interpretazione come fiberwise configurations} we have
		\[
		H_*^{S^1}(C_n(\C);\F_2)= H_{*+2qn}(D_n(\lambda;S^{2q});\F_2)
		\]
		By Theorem \ref{thm:stable splitting} we get
		\[
		\Tilde{H}_*(E(\lambda;S^{2q});\F_2)=\bigoplus_{n=1}^{\infty}\Tilde{H}_*(D_n(\lambda;S^{2q});\F_2)
		\]
		so $H_*^{S^1}(C_n(\C);\F_2)$ can be seen as a subspace of $H_*(E(\lambda;S^{2q});\F_2)$. Finally, the Theorem \ref{thm: calcolo fiberqise configurations nel caso p=2} tell us that 
		\[
		H_*(E(\lambda;S^{2q});\F_2)=\F_2[\iota,Q\iota,Q^2\iota,\dots]\otimes H_*(BS^1;\F_2)
		\]
		where $\iota\in H_{2q}(S^{2q});\F_2)$ is the fundamental class, and this is enough to get the statement.
	\end{proof}
	\subsection{Examples}\label{sec: examples}
	Let $p$ be an odd prime. We end this paper with the description of an additive basis of $H_*(B_n/Z(B_n);\F_p(\pm 1))$ and with some examples in the case $p=3$ and $n\leq 5$. Proposition \ref{prop: interpretazione come fiberwise configurations} and Theorem \ref{thm:stable splitting} tells us that $H_*(B_n/Z(B_n);\F_p(\pm 1))$ can be identified (up to a shift) with a subspace of $H_*(E(\lambda;S^{2q+1});\F_p)\cong H_*(C(\C;S^{2q+1});\F_p)\otimes H_*(BS^1;\F_p)$ (Theorem \ref{thm:collasso sequenza spettrale per i labelled fiberwise config. spaces}). In order to identify explicitly a basis of $H_*(B_n/Z(B_n);\F_p(\pm 1))$ inside $H_*(E(\lambda;S^{2q+1});\F_p)$ we need to introduce some notation. 		
	Recall that there is a filtration (see Section \ref{sec: fiberwise configuration space})
	\[
	\{\ast\}\subseteq C_1(\C;X)\subseteq C_2(\C;X)\subseteq\dots\subseteq C(\C;X)
	\]
	of the labelled configuration space $C(\C;X)$, where $C_k(\lambda;X)$ is the subspace 
	\[
	C_k(\C;X)\coloneqq \bigsqcup_{n=0}^{k}F_n(\C)\times_{\Sigma_n}X^n/\sim 
	\]
	We will denote the quotient $C_k(\C;X)/C_{k-1}(\C;X)$ by $D_k(\C;X)$. By Theorem \ref{thm: omologia spazi di configurazione etichettati} $H_*(C(\C;S^{2q+1});\F_p)$ is the free graded commutative algebra on classes $Q^i(x)$ and $\beta Q^i(x)$, where $x$ is a basic bracket of odd degree. In this case the only basic bracket is the fundamental class $x\in H_{2q+1}(S^{2q+1})$, indeed $[x,x]=0$ since $x$ is a odd degree variable. The decomposition
	\[
	\Tilde{H}_*(C(\C;S^{2q+1}))\cong\bigoplus_{k=1}^{\infty}\Tilde{H}_*(D_k(\C;S^{2q+1}))
	\]
	induced by the filtration of $C(\C;S^{2q+1})$ is encoded by a weight grading $w$, which is generated by the following rules:
	\begin{itemize}
		\item $w(1)=0$
		\item $w(x)=1$
		\item $w(ab)=w(a)+w(b)$
		\item $w(Q(a))=w(\beta Q(a))=pw(a)$
	\end{itemize} 
	The subspace $\Tilde{H}_*(D_k(\C;S^{2q+1});\F_p)$ is generated by the monomials of weight $k$. The following Theorem identifies explicitly $H_*(B_n/Z(B_n);\F_p(\pm 1))$ as a subspace of $H_*(E(\lambda,S^{2q+1});\F_p)$ by describing an additive basis. 
	\begin{thm}\label{thm: identificazione esplicita dell'omologia twistata}
		Let $p$ be an odd prime, $n\geq 2$.
		\[                                                  
		s^{(2q+1)n}H_*(B_n/Z(B_n);\F_p(\pm 1))\cong \Tilde{H}_{*}(D_n(\C;S^{2q+1});\F_p)\otimes \F_p[c]
		\]
		where $c$ is a variable of degree $2$. Therefore $s^{(2q+1)n}H_*(B_n/Z(B_n);\F_p(\pm 1))$ is generated by elements of the form
		\[
		a\otimes c^i \qquad i\geq 0
		\] 
		where $a\in H_*(C(\C;S^{2q+1});\F_p)$ is a monomial of weight $n$. 
	\end{thm}
	\begin{proof}
		Combining Proposition \ref{prop: interpretazione come fiberwise configurations} and Theorem \ref{thm:stable splitting} we get that
		\[
		s^{(2q+1)n}H_*(B_n/Z(B_n);\F_p(\pm 1))\cong \Tilde{H}_{*}(D_n(\lambda;S^{2q+1});\F_p)\subseteq \Tilde{H}_*(E(\lambda;S^{2q+1});\F_p)
		\]
		By Theorem \ref{thm:collasso sequenza spettrale per i labelled fiberwise config. spaces} we have
		\[
		H_*(E(\lambda;S^{2q+1});\F_p)\cong H_*(C(\C;S^{2q+1});\F_p)\otimes H_*(BS^1;\F_p)
		\]
		and this isomorphism is compatible with the filtrations of $E(\lambda;S^{2q+1})$ and $C(\C;S^{2q+1})$. Therefore $\Tilde{H}_{*}(D_n(\lambda;S^{2q+1});\F_p)$ corresponds to $\Tilde{H}_{*}(D_n(\C;S^{2q+1});\F_p)\otimes H_*(BS^1;\F_p)$, and this proves the statement. 
	\end{proof}
	We end this paper with some computations in the case $p=3$ and $n\leq 5$. From now on we fix $q=0$: in this case the weights and degrees of the generators $Q^i(x)$, $\beta Q^i(x)$ are summarized in this table:
	\[
	\begin{tabular}{ccc}
		
		\toprule
		Homology class  & Weight & Degree\\
		\midrule
		$x$ & $1$ & $0$ \\
		$Q^i(x)$ & $p^i$ & $2p^i-1$ \\
		$\beta Q^i(x)$ & $p^i$ & $2p^i-2$ \\
		\bottomrule
	\end{tabular}
	\]
	\begin{description}
		\item[$n=2$:] By Theorem \ref{thm: identificazione esplicita dell'omologia twistata} we have 
		\[
		s^2H_*(B_2/Z(B_2);\F_3(\pm 1))\cong \Tilde{H}_{*}(D_2(\C;S^1);\F_3)\otimes \F_3[c]
		\]
		Since there are no classes of weight $2$ inside $H_*(C(\C;S^1);\F_3)$ we get that 
		\[
			H_*(B_2/Z(B_2);\F_3(\pm 1))=0
		\]
		\item[$n=3$:] By Theorem \ref{thm: identificazione esplicita dell'omologia twistata} we have 
	\[
	s^3H_*(B_3/Z(B_3);\F_3(\pm 1))\cong \Tilde{H}_{*}(D_3(\C;S^1);\F_3)\otimes \F_3[c]
	\]
The only classes of weight $3$ inside $\Tilde{H}_*(C(\C;S^1);\F_3)$ are $Q(x)$ and $\beta Q(x)$. Therefore $s^3H_*(B_3/Z(B_3);\F_3(\pm 1))$ is generated as $\F_3$ vector space by classes 
\[
Q(x)\otimes c^i \qquad \beta Q(x)\otimes c^i\qquad i\geq 0
\]
where $c$ has degree two, $\beta Q(x)$ has degree four and $Q(x)$ has degree five. In other words,
\[
H_*(B_3/Z(B_3);\F_3(\pm 1))=\begin{cases}
	\F_3 \text{ if } *\geq 1\\
	0 \text{ if } *=0
\end{cases}
\] 
	\item[$n=4$:] By Theorem \ref{thm: identificazione esplicita dell'omologia twistata} we have 
\[
s^4H_*(B_4/Z(B_4);\F_3(\pm 1))\cong \Tilde{H}_{*}(D_4(\C;S^1);\F_3)\otimes \F_3[c]
\]
The only classes of weight $4$ inside $\Tilde{H}_*(C(\C;S^1);\F_3)$ are $xQ(x)$ and $x\beta Q(x)$. Therefore $s^4H_*(B_4/Z(B_4);\F_3(\pm 1))$ is generated as $\F_3$ vector space by classes 
\[
xQ(x)\otimes c^i \qquad x\beta Q(x)\otimes c^i\qquad i\geq 0
\]
where $c$ has degree two, $x\beta Q(x)$ has degree five and $xQ(x)$ has degree six. In other words,
\[
H_*(B_4/Z(B_4);\F_3(\pm 1))=\begin{cases}
	\F_3 \text{ if } *\geq 1\\
	0 \text{ if } *=0
\end{cases}
\] 
	\item[$n=5$:] By Theorem \ref{thm: identificazione esplicita dell'omologia twistata} we have 
\[
s^5H_*(B_5/Z(B_5);\F_3(\pm 1))\cong \Tilde{H}_{*}(D_5(\C;S^1);\F_3)\otimes \F_3[c]
\]
Since there are no classes of weight $5$ inside $H_*(C(\C;S^1);\F_3)$ we get that 
\[
H_*(B_5/Z(B_5);\F_3(\pm 1))=0
\]
	\end{description}

	\printbibliography
	
\end{document}